\documentclass[a4paper,12pt,times]{amsproc}

\usepackage{geometry}
\usepackage{bm}
\usepackage{amssymb}
\usepackage{graphicx}
\usepackage{mathrsfs}
\usepackage[mathscr]{eucal}
\usepackage{natbib}
\usepackage{booktabs}
\usepackage{epstopdf}
\usepackage{enumerate}
\usepackage[colorlinks,citecolor=blue]{hyperref}
\usepackage{adjustbox}
\usepackage{caption}
\usepackage{subcaption}

\numberwithin{equation}{section}

\def\bW{{\mathbf W}}
\def\bB{{\mathbf B}}
\def\bA{{\mathbf A}} \def\bC{{\mathbf C}}  
 \def\bS{{\mathbf S}} 
 \def\bH{{\mathbf H}} \def\bK{{\mathbf K}} \def\bI{{\mathbf I}}
   \def\bQ{{\mathbf Q}}
 \def\bS{{\mathbf S}}  
 \def\bW{{\mathbf W}} \def\bX{{\mathbb X}} \def\bY{{\mathbb Y}}

\def\bM{\bm M}

\def\bx{{\mathbf x}} \def\by{{\mathbf y}} 

\def\bbeta{{\bm{\beta}}}

\def\hbbeta{\widehat{\bm \beta}}
\def\hsigma{\widehat{\sigma}}

\def\argmin{\mathop{\rm arg\, min}}

\def\real{\mathop{{\rm I}\kern-.2em\hbox{\rm R}}\nolimits}

\def\1overn{\frac{1}{n}}

\def\bel{\begin{eqnarray}\label}  \def\eel{\end{eqnarray}}
\def\bes{\begin{eqnarray*}}  \def\ees{\end{eqnarray*}}

\newtheorem{theorem}{Theorem}[section]

\theoremstyle{definition}

\theoremstyle{remark}

\begin{document}

\title[Shrinkage Estimation in Generalized Ridge Regression]{Shrinkage Estimation Strategies in Generalized Ridge Regression Models Under Low/High-Dimension Regime}


\author{Bahad{\i}r Y\"{u}zba\c{s}{\i}$^\dag$, Mohammad Arashi$^\ddag$ \and S.Ejaz Ahmed$^\S$}

\date{\today}
\maketitle

{\footnotesize
\center { \text{  $^\dag$Department of Econometrics}\par
  { \text{ Inonu University, Turkey}}\par
  { \texttt{E-mail address:b.yzb@hotmail.com}}
  \par

  \vskip 0.2 cm
  
  \text{  $^\ddag$Department of Statistics, Shahrood University of Technology, Iran}\par
  { \text{Department of Statistic, University of Pretoria, South Africa}}\par
  { \texttt{E-mail address:m\_arashi\_stat@yahoo.com}}

  \vskip 0.2 cm
  
  \text{  $^\S$Department of Mathematics and Statistics}\par
  { \text{Brock University, Canada}}\par
  { \texttt{E-mail address:sahmed5@brocku.ca}}

}}

\renewcommand{\thefootnote}{}

\footnote {\par

 Corresponding author : Bahad{\i}r Y\"{u}zba\c{s}{\i} \par

}

\date{\today}

\begin{abstract}

In this study, we propose shrinkage methods based on {\it generalized ridge regression} (GRR) estimation which is suitable for both multicollinearity and high dimensional problems with small number of samples (large $p$, small $n$). Also, it is obtained theoretical properties of the proposed estimators for Low/High Dimensional cases. Furthermore, the performance of the listed estimators is demonstrated by both simulation studies and real-data analysis, and compare its performance with existing penalty methods. We show that the proposed methods compare well to competing regularization techniques.

\end{abstract}

\maketitle

\section{Introduction}
\noindent

In a multiple linear regression model is used by data analysts in nearly every field of science and technology as well as economics, econometrics, finance, medicine, engineering. In such models,  however, the predictors may or may not dependent of each other. When there exists the near-linear relationships among the predictors, the problem of multicollinearity or collinearity occurs. Furthermore, the {\it least squares estimate} (LSE) will be unbiased, but its variances is large so it may be far from the true value. In order to deal with this problem, one of the best way, in literature, is the {\it Ordinary Ridge Regression} (ORR) by \citet{Ho-Ke1970} technique which adds a degree of bias to the regression estimates for reducing the standard errors.

On the other hand, along with technological innovations, it is not unusual to see the number of covariates ($p$) greatly exceed the number of observations ($n$), which is known as ``large $p$, small $n$'' problem, e.g. micro-array data analysis, environmental pollution studies. The availability of massive data has occurred new problems in statistical analysis. With many predictors, fitting the full model without penalization will result in large prediction intervals, and the LSE may not uniquely exist. Many approaches have been proposed for tackling this problem such as the ORR \citet{Ho-Ke1970}, the {\it Elastic-Net} (Enet) method \citet{ZouHastie2005}, the {\it least absolute shrinkage and selection operation} (Lasso) \index{Lasso} \citet{Tibshirani1996}, the {\it smoothly clipped
absolute deviation method} (SCAD) \citet{FanLi2001}\index{SCAD}, the {\it adaptive Lasso} (aLasso)
\citet{Zou2006}\index{Adaptive Lasso}, and {\it minimax concave
penalty} (MCP) \citet{Zhang2010}\index{MCP} and the {\it least angle regression} \citet{Efron2004} are some approaches.

\citet{saleh2006} and \citet{Ahmed2014} are excellent sources for a comprehensive overview on shrinkage estimation with uncertain prior information. \citet{rozAras2016} and \citet{bahadir-ahmed2016} developed ride type shrinkage estimation in {\it partial linear models} (PLMs). Also, \citet{wuasar2016} considered a weighted stochastic restricted ridge estimator in such models. Very recently, \citet{yuzbasi-et-al2017} proposed  pretest and shrinkage estimators in ridge regression linear models and compare their performance with some penalty estimators, including some penalty estimators.

In high-dimensional setting, \citet{Ishwaran_Rao2014} demonstrated a geometric approach to study the properties of GRR. \citet{bahadir-ahmed2015} provided some numerical comparison of shrinkage ridge regression estimators. \citet{AhmedYzb2016,AhmedYzb2017} and \citet{yuzbasi-et-al2017b} suggested a Stein-type shrinkage estimation strategy via integration penalty estimators. \citet{gao-et-al2017} proposed a post selection shrinkage estimation strategy based on weighted ridge estimator.


In this paper, we propose a number of shrinkage estimation strategies by using GRR method which is an extension of RR in low/high-dimensional linear models.

In section \ref{sec:SM}, along with the preliminary notions, full and restricted estimators are introduced. All the proposed estimators are defined, and also their theoretical properties are obtained in Section \ref{sec:PrE}. Furthermore, we introduce some shrinkage estimators in the context of high dimensional case in Section \ref{sec:HD}. It is demonstrated the asymptotic properties of listed estimators in this section. In Section \ref{sec:PE}, the penalized estimators, which are used in this paper, are given. A detailed Monte Carlo simulation study are conducted in Section \ref{sec:SS}. Also, two real-life applications are analyzed in Section \ref{sec:RDA}. Finally, Section \ref{sec:con} gives concluding remarks.

\section{Statistical Model}
\label{sec:SM}
\noindent
Consider the following linear model
\begin{equation}\label{full:model}
\bY=\bX \bm\beta+ \bm \varepsilon, \quad p < n
\end{equation}
where $\bY=(y_1,\dots,y_n)\in\mathbb{R}^{n}$ is a random vector of response variables, $\bX=(\bx_1\dots,\bx_n)'\in\mathbb{R}^{n\times p}$ is the design matrix of full column rank with $\bx_i=(x_{i1},\dots,x_{ip})'$ for $i=1,\dots,n$, $\bbeta=(\beta_1,\dots,\beta_p)'\in\mathbb{R}^{p}$ is unknown parameters and $\bm \varepsilon=(\varepsilon_1,\dots,\varepsilon_n)'\in\mathbb{R}^{n}$ is the random error vector such that $\mathbb{E}(\varepsilon_i) = 0$ and $\mathbb{V}(\varepsilon_i)=\sigma_0^2>0$.

For the model \ref{full:model}, the LSE of $\bbeta$ is given by
\begin{equation}
\bm{\widehat{\beta}}_n^{\textrm{\rm LSE}}=\left( \bX'\bX\right)^{-1}\bX'\bY=\bS^{-1}\bX'\bY,
\end{equation}
which is the {\it best linear unbiased estimator} (BLUE) of the coefficient $\bbeta$ under the Gauss--Markov theorem. 
The LSE estimator loses its efficiency, however, when the multicollinearity is occur in the model \ref{full:model}. To tackle this problem,
there is a large literature of estimators. One of the effective estimations among them, \cite{Ho-Ke1970} introduced 
the ORR estimation of $\bbeta$, which has lower the {\it mean square error} (MSE) than the LSE, is given by
\begin{equation}
\bm{\widehat{\beta}}_n^{\textrm{ORR}}=\left( \bX'\bX+k\bI_p\right)^{-1}\bX'\bY=\bS_k^{-1}\bX'\bY,
\end{equation}
where $k>0$ is the ridge tuning parameter in order to tend to reduce the magnitude of the estimated regression
coefficients, leading to fewer effective model parameters. The ORR can be helpful for avoiding over--fitting and in improving numerical stability in the case of an
ill-conditioned $\bS$ matrix.  The GRR estimation of $\bbeta$ is given by
\begin{equation}
\bm{\widehat{\beta}}_n^{\textrm{\rm GRR}}=\left( 
\bS+\bK\right)^{-1}\bX'\bY=\bS_{\bK}^{-1}\bX'\bY,
\label{equ:GRR}
\end{equation}
where $\bK=diag(k_1,\dots,k_p)$. In the GRR, $p$ ridge parameters need to be determined while the ORR requires determination of one parameter only. The GRR can have outstanding performance the ORR if the ridge matrix is appropriately selected. Hence, the GRR is more conservative, outperforms the ORR in view of weighted quadratic loos measure. \cite{Ho-Ke1970} suggest to select $k_j=\sigma^2/\bbeta_j^2$, $j=1,\dots,p$.

In this study, our primary interest is to estimate the regression coefficient $\bbeta$ when it is a priori suspected
but not certain that $\bbeta$ may be restricted to the subspace
\begin{equation}
\rm H_0: {\bH}\bm{\beta }=\bold0, \label{UPI}
\end{equation}
where $\bH$ is a $q\times p$ known matrix of full rank $q(<p)$.

When the null hypothesis it true, we estimate the restricted GRR (RGRR) estimator of $\bbeta$ is given by
\begin{equation}
\bm{\widehat{\beta}}_n^{\textrm{\rm RGRR}}=\left( \bI+\bS_\bK^{-1}\bH'(\bH\bS_\bK^{-1}\bH')^{-1}\bH\right)\bm{\widehat{\beta}}_n^{\textrm{\rm GRR}}=\bM_\bK\bm{\widehat{\beta}}_n^{\textrm{\rm GRR}}.
\label{equ:RGRR}
\end{equation}

\section{Proposed estimators}
\label{sec:PrE}
Now, we consider some shrinkage and pretest strategies to combine the information from GRR and RGRR.
First, we define the {\it Linear Shrinkage} (LS) estimator of $\bm{\beta}$ as follows:
\begin{eqnarray}
\hbbeta_n^{\textrm{LS}}&=&\omega \hbbeta_n^{\textrm{\rm RGRR}}+\left (1-\omega \right )\hbbeta_n^{\textrm{\rm GRR}}, \\ \label{LS}
& = & \hbbeta_n^{\textrm{\rm GRR}} - \omega \left (\hbbeta_n^{\textrm{\rm GRR}}-\hbbeta_n^{\textrm{\rm RGRR}} \right ) \nonumber
\end{eqnarray}
where $\omega \in [0, 1]$ denotes the shrinkage intensity. Ideally, the coefficient $\omega$ is chosen to minimize the mean squared error. This estimator suffers from the lack of motivation. In other words, one may suspect the hypothesis $H_0:\bH\bbeta=\bold0$ holds in application. Thus, we can decide on the use of any $\hbbeta_n^{\textrm{\rm GRR}}$ or $\hbbeta_n^{\textrm{\rm RGRR}}$ depending the output of testing $H_0$. From \citet{saleh2006}, we use $\mathscr{W}_{n}$ as an appropriate test statistic to test \eqref{UPI}, which is defined by
\begin{equation*}
\mathscr{W}_{n}=\frac{1}{q\widehat{\sigma}^{2}} (\bH{\hbbeta}_n^{\textrm{\rm LSE}})'(\bH \bS^{-1}\bH')^{-1}(\bH{\hbbeta}_n^{\textrm{\rm LSE}}),
\end{equation*}
where $\widehat{\sigma}^{2}=\frac{1}{m}(\bY-\bX\bm{\widehat{\beta}}_n^{\textrm{\rm LSE}})'(\bY-\bX\bm{\widehat{\beta}}_n^{\textrm{\rm LSE}})$, $m=n-p$.

Using Theorem 7.1.2.1 of \citet{saleh2006}, under $H_0$, $\mathscr{W}_{n}$ has the $\mathscr{F}$-distribution with $(q,m)$ {\it degrees of freedom} (d.f.) and under the alternative hypothesis $H_A:{\bH}\bm{\beta}\neq{\bold0}$, it has the non-central $\mathscr{F}$-distribution with the same d.f. and non-centrality parameter $\Delta^2=(\bH\bbeta)'(\bH \bS^{-1}\bH')^{-1}(\bH\bbeta)/\sigma^2$.

Now, in general, we can formulate shrinkage estimators as
\begin{equation*}
{\hbbeta}_n^{\textrm{Shrinkage}}=\hbbeta_n^{\textrm{\rm GRR}}-\left (\hbbeta_n^{\textrm{\rm GRR}}-\hbbeta_n^{\textrm{\rm RGRR}}  \right )g(\mathscr{W}_n),
\end{equation*}
where $g(\cdot)$ is a suitably chosen Borel measurable function. Specific choices of $g(\cdot)$ will give reasonable and useful shrinkage estimators. For $g(\mathscr{W}_n)=0$, we get $\hbbeta_n^{\rm GRR}$ while for $g(\mathscr{W}_n)=1$ we get $(\hbbeta_n^{\rm RGRR})'$. In what follows, we list some other candidates.

The {\it preliminary test} (PT) estimator of $\bm{\beta}$ is obtained by taking $g(\mathscr{W}_{n})=\textrm{I}\left(\mathscr{W}_{n}\leq  \mathscr{F}_{q,m}(\alpha)\right)$ as
\begin{equation}
\label{beta_PT}
{\hbbeta}_n^{\textrm{PT}}=\hbbeta_n^{\textrm{\rm GRR}}-\left (\hbbeta_n^{\textrm{\rm GRR}}-\hbbeta_n^{\textrm{\rm RGRR}}  \right ) \textrm{I}\left(\mathscr{W}_{n}\leq  \mathscr{F}_{q,m}(\alpha)\right),
\end{equation}%
where $\textrm{I}\left(A\right)$ is the indicator function of the set $A$.

\citet{Ahmed1992} proposed a {\it shrinkage pretest estimation} (SPT) strategy replacing $\hbbeta_n^{\textrm{\rm RGRR}}$ by $\hbbeta_n^{\textrm{LS}}$ in \eqref{beta_PT} as follows:
\begin{equation*}
\hbbeta_n^{\textrm{SPT}}= \hbbeta_n^{\textrm{\rm GRR}}-\omega\left (\hbbeta_n^{\textrm{\rm GRR}}-\hbbeta_n^{\textrm{\rm RGRR}}  \right )\textrm{I}\left(\mathscr{W}_{n}\leq  \mathscr{F}_{q,m}(\alpha)\right).
\end{equation*}
In this case, $g(\mathscr{W}_n)=\omega \textrm{I}\left(\mathscr{W}_{n}\leq  \mathscr{F}_{q,m}(\alpha)\right)$.

Taking $g(\mathscr{W}_{n})=d\mathscr{W}_n^{-1}$, yields the {\it Stein-type shrinkage} (S) estimator, which is a combination of the overfitted model estimator $\bm{\widehat{\beta}}_n^{\textrm{\rm GRR}}$ with the under-fitted $\bm{\widehat{\beta}}_n^{\textrm{\rm RGRR}}$, given by
\begin{equation*}
\bm{\widehat{\beta}}_n^{\textrm{S}}=\bm{\widehat{\beta}}_n^{\textrm{\rm GRR}}-d\left( \bm{\widehat{\beta}}_n^{\textrm{\rm GRR}}-\bm{\widehat{\beta}}_n^{\textrm{\rm RGRR}}\right)\mathscr{W}_{n}^{-1} \text{, } d=(q-2)m/q(m+2),
\end{equation*}

In an effort to avoid the over-shrinking problem inherited by $\bm{\widehat{\beta}}_n^{\textrm{S}}$ we suggest using the {\it positive part} (PS) of the S estimator defined by
\begin{eqnarray*}
\hbbeta_n^{\textrm{PS}}&=&\hbbeta_n%
^{\textrm{\rm RGRR}}+\left( 1-d\mathscr{W}_{n}^{-1}\right)\textrm{I}\left(\mathscr{W}_{n} >  d\right)\left( \hbbeta_n^{\textrm{\rm GRR}}-\bm{%
\widehat{\beta }}_n^{\textrm{\rm RGRR}}\right)\cr
&=&\bm{\widehat{\beta}}_n^{\textrm{S}}-\left( \hbbeta_n^{\textrm{\rm GRR}}-\bm{%
\widehat{\beta }}_n^{\textrm{\rm RGRR}}\right)
\left( 1-d\mathscr{W}_{n}^{-1}\right)\textrm{I}\left(\mathscr{W}_{n}\leq  d\right)
\end{eqnarray*}%
where we used $g(\mathscr{W}_n)=d\mathscr{W}_n^{-1}+\left( 1-d\mathscr{W}_{n}^{-1}\right)\textrm{I}\left(\mathscr{W}_{n}\leq  d\right)$.

If we replace $\hbbeta_n^{\textrm{\rm GRR}}$ by $\hbbeta_n^{\textrm{S}}$ in $\hbbeta_n^{\textrm{PT}}$, we obtain the {\it improved pretest} (IPT) estimator of $\bm{\beta}$ defined by
\begin{eqnarray*}
\hbbeta_n^{\textrm{IPT}}&=&\hbbeta_n^{\textrm{S}}-\left (\hbbeta_n^{\textrm{S}}-\hbbeta_n^{\textrm{\rm RGRR}}  \right ) \textrm{I}\left(\mathscr{W}_{n}\leq  \mathscr{F}_{q,m}(\alpha)\right), \\
&=& \hbbeta_n^{\textrm{S}}- \left (\hbbeta_n^{\textrm{\rm GRR}}-\hbbeta_n^{\textrm{\rm RGRR}}  \right )
\left ( 1- d\mathscr{W}_{n}^{-1}\right )%
\textrm{I}\left(\mathscr{W}_{n} \leq  \mathscr{F}_{q,m}(\alpha)\right),
\end{eqnarray*}%
where we used $g(\mathscr{W}_n)=d\mathscr{W}_n^{-1}+\left( 1-d\mathscr{W}_{n}^{-1}\right)\textrm{I}\left(\mathscr{W}_{n}\leq  \mathscr{F}_{q,m}(\alpha)\right)$.

We also consider the performance of the listed methods in the following section.

\subsection{Properties}
Let $\bbeta_n^*$ denote any of the five estimators of $\bbeta$ defined in previous section. The bias of $\bbeta_n^*$ is defined by ${\rm\bB}(\bbeta_n^*)=\mathbb{E}(\bbeta^*-\bbeta)$. Considering the loss function ${\rm L}(\bbeta;\bbeta_n^*)=(\bbeta_n^*-\bbeta)'\bW(\bbeta_n^*-\bbeta)$, for a suitably positive definite weight matrix $\bW$, the associated quadratic risk function is defined by ${\rm R}(\bbeta_n^*)=\mathbb{E}({\rm L}(\bbeta;\bbeta_n^*))$. In the following results we give the bias and quadratic risk functions of the proposed five estimators of $\bbeta$.
\begin{theorem}
\label{bias} 
The bias of the estimators $\hbbeta_n^{\rm GRR}$, $\hbbeta_n^{\rm RGRR}$, $\hbbeta_n^{\rm LS}$, $\hbbeta_n^{\rm PT}$, $\hbbeta_n^{\rm SPT}$, $\hbbeta_n^{\rm S}$, $\hbbeta_n^{\rm PS}$ and $\hbbeta_n^{\rm IPT}$ are respectively given by
\begin{eqnarray*}
{\rm\bB}(\hbbeta_n^{\rm GRR})&=&(\bS_{\bK}^{-1}\bS-\bI_p)\bbeta\cr
{\rm\bB}(\hbbeta_n^{\rm RGRR})&=&(\bM_{\bK}\bS_{\bK}^{-1}\bS-\bI_p)\bbeta\cr
{\rm\bB}(\hbbeta_n^{\rm LS})&=&(\bS_{\bK}^{-1}\bS-\bI_p)\bbeta-\omega(\bI_p-\bM_{\bK})\bS_{\bK}^{-1}\bS\bbeta\cr
{\rm\bB}(\hbbeta_n^{\rm PT})&=&\bS_{\bK}^{-1}\bS\bbeta\left[1-\mathbb{G}_{q+2,m}(c_\alpha;\Delta^2)\right]\cr
&&+\bM_{\bK}\bS_{\bK}^{-1}\bS\bbeta \mathbb{G}_{q+2,m}(c_\alpha;\Delta^2)\cr
{\rm\bB}(\hbbeta_n^{\rm SPT})&=&\bS_{\bK}^{-1}\bS\bbeta\left[1-\omega \mathbb{G}_{q+2,m}(c_\alpha;\Delta^2)\right]\cr
&&+\omega\bM_{\bK}\bS_{\bK}^{-1}\bS\bbeta G_{q+2,m}(c_\alpha;\Delta^2)\cr
{\rm\bB}(\hbbeta_n^{\rm S})&=&\bS_{\bK}^{-1}\bS\bbeta\left[1-d \mathbb{E}(\mathscr{F}^{-1}_{q+2,m}(\Delta^2))\right]\cr
&&+d\bM_{\bK}\bS_{\bK}^{-1}\bS\bbeta \mathbb{E}(\mathscr{F}^{-1}_{q+2,m}(\Delta^2))\cr
{\rm\bB}(\hbbeta_n^{\rm PS})&=&{\rm\bB}(\hbbeta_n^{\rm S})-(\bI_p-\bM_{\bK})\bS_{\bK}^{-1}\bS\mathbb{G}_{q+2,m}(d_1;\Delta^2)\cr
&&+d_1(\bI_p-\bM_{\bK})\bS_{\bK}^{-1}\bS\bbeta\mathbb{E}[\mathscr{F}^{-1}_{q+2,m}\textrm{I}(\mathscr{F}_{q+2,m}(\Delta^2)\leq d_1)]\cr
{\rm\bB}(\hbbeta_n^{\rm IPT})&=&{\rm\bB}(\hbbeta_n^{\rm S})-(\bI_p-\bM_{\bK})\bS_{\bK}^{-1}\bS\mathbb{G}_{q+2,m}(c_\alpha;\Delta^2)\cr
&&+d(\bI_p-\bM_{\bK})\bS_{\bK}^{-1}\bS\bbeta\mathbb{E}[\mathscr{F}^{-1}_{q+2,m}\textrm{I}(\mathscr{F}_{q+2,m}(\Delta^2)\leq \mathscr{F}_{q+2,m}(\alpha))],
\end{eqnarray*}
where $\mathbb{G}_{\vartheta_1,\vartheta_2}(\cdot;\Delta^2)$ is the cdf of a non-central $\mathscr{F}$-distribution with $(\vartheta_1, \vartheta_2)$ degree of freedom and non-centrality parameter $\Delta^2/2$.
\end{theorem}
For the proof, see the Appendix.

\begin{theorem} 
\label{risk}
The quadratic risk function of the estimators $\hbbeta_n^{\rm GRR}$, $\hbbeta_n^{\rm RGRR}$, $\hbbeta_n^{\rm LS}$, $\hbbeta_n^{\rm PT}$, $\hbbeta_n^{\rm SPT}$, $\hbbeta_n^{\rm S}$, $\hbbeta_n^{\rm PS}$ and $\hbbeta_n^{\rm IPT}$ are respectively given by
\begin{eqnarray*}
{\rm R}(\hbbeta_n^{\rm GRR})&=&\sigma^2{\rm tr}\left(\bS_{\bK}^{-1}\bS\bS_{\bK}^{-1}\bW\right)+\bbeta'(\bS_{\bK}^{-1}\bS-\bI_p)\bW(\bS\bS_{\bK}^{-1}-\bI_p)\bbeta\cr
{\rm R}(\hbbeta_n^{\rm RGRR})&=&\sigma^2{\rm tr}\left(\bS_{\bK}^{-1}\bS\bS_{\bK}^{-1}\bW\right)+\bbeta'(\bS_{\bK}^{-1}\bS-\bI_p)\bW(\bS\bS_{\bK}^{-1}-\bI_p)\bbeta\cr
&&-2{\rm tr}(\bB)-2\bbeta'\bB\bbeta+2\bbeta'\bW(\bI_p-\bM_{\bK})\bS_{\bK}^{-1}\bS\bbeta\cr
&&+{\rm tr}(\bC)+\bbeta'\bC\bbeta\cr
{\rm R}(\hbbeta_n^{\rm LS})&=&\sigma^2{\rm tr}\left(\bS_{\bK}^{-1}\bS\bS_{\bK}^{-1}\bW\right)+\bbeta'(\bS_{\bK}^{-1}\bS-\bI_p)\bW(\bS\bS_{\bK}^{-1}-\bI_p)\bbeta\cr
&&-2\omega{\rm tr}(\bB)-2\omega\bbeta'\bB\bbeta+2\omega\bbeta'\bW(\bI_p-\bM_{\bK})\bS_{\bK}^{-1}\bS\bbeta\cr
&&+\omega^2{\rm tr}(\bC)+\omega^2\bbeta'\bC\bbeta\cr
{\rm R}(\hbbeta^{\rm  PT})&=&\sigma^2{\rm tr}\left(\bS_{\bK}^{-1}\bS\bS_{\bK}^{-1}\bW\right)+\bbeta'(\bS_{\bK}^{-1}\bS-\bI_p)\bW(\bS\bS_{\bK}^{-1}-\bI_p)\bbeta\cr
&&-2{\rm tr}(\bB)\mathbb{G}_{q+2,m}(c_\alpha;\Delta^2)-2\bbeta'\bB\bbeta\mathbb{G}_{q+4,m}(c_\alpha;\Delta^2)\cr
&&+2\bbeta'\bW(\bI_p-\bM_{\bK})\bS_{\bK}^{-1}\bS\bbeta\mathbb{G}_{q+2,m}(c_\alpha;\Delta^2)\cr
&&+{\rm tr}(\bC)\mathbb{G}_{q+2,m}(c_\alpha;\Delta^2)+\bbeta'\bC\bbeta\mathbb{G}_{q+4,m}(c_\alpha;\Delta^2)\cr
{\rm R}(\hbbeta_n^{\rm SPT})&=&\sigma^2{\rm tr}\left(\bS_{\bK}^{-1}\bS\bS_{\bK}^{-1}\bW\right)+\bbeta'(\bS_{\bK}^{-1}\bS-\bI_p)\bW(\bS\bS_{\bK}^{-1}-\bI_p)\bbeta\cr
&&-2\omega{\rm tr}(\bB)\mathbb{G}_{q+2,m}(c_\alpha;\Delta^2)-2\omega\bbeta'\bB\bbeta\mathbb{G}_{q+4,m}(c_\alpha;\Delta^2)\cr
&&+2\omega\bbeta'\bW(\bI_p-\bM_{\bK})\bS_{\bK}^{-1}\bS\bbeta\mathbb{G}_{q+2,m}(c_\alpha;\Delta^2)\cr
&&+\omega^2{\rm tr}(\bC)\mathbb{G}_{q+2,m}(c_\alpha;\Delta^2)+\omega^2\bbeta'\bC\bbeta\mathbb{G}_{q+4,m}(c_\alpha;\Delta^2)\cr
{\rm R}(\hbbeta_n^{\rm S})&=&\sigma^2{\rm tr}\left(\bS_{\bK}^{-1}\bS\bS_{\bK}^{-1}\bW\right)+\bbeta'(\bS_{\bK}^{-1}\bS-\bI_p)\bW(\bS\bS_{\bK}^{-1}-\bI_p)\bbeta\cr
&&-2d{\rm tr}(\bB)\mathbb{E}[\mathscr{F}^{-1}_{q+2,m}(\Delta^2)]-2d\bbeta'\bB\bbeta\mathbb{E}[\mathscr{F}^{-1}_{q+4,m}(\Delta^2)]\cr
&&+2d\bbeta'\bW(\bI_p-\bM_{\bK})\bS_{\bK}^{-1}\bS\bbeta\mathbb{E}\left[\mathscr{F}^{-1}_{q+2,m}(\Delta^2)\right]\cr
&&+d^2{\rm tr}(\bC)\mathbb{E}[\mathscr{F}^{-2}_{q+2,m}(\Delta^2)]+d^2\bbeta'\bC\bbeta\mathbb{E}[\mathscr{F}^{-2}_{q+4,m}(\Delta^2)]\cr
{\rm R}(\hbbeta_n^{\rm PS})&=&{\rm R}(\hbbeta_n^{\rm S})\cr
&&+\sigma^2{\rm tr}\left(\bS_{\bK}^{-1}\bS\bS_{\bK}^{-1}\bW\right)+\bbeta'(\bS_{\bK}^{-1}\bS-\bI_p)\bW(\bS\bS_{\bK}^{-1}-\bI_p)\bbeta\cr
&&-2{\rm tr}(\bB)\mathbb{E}[\left(1-d\mathscr{F}^{-1}_{q+2,m}(\Delta^2)\right)I(\mathscr{F}^{-1}_{q+2,m}(\Delta^2)\leq d_1)]\cr
&&-2\bbeta'\bB\bbeta\mathbb{E}[\left(1-d\mathscr{F}^{-1}_{q+4,m}(\Delta^2)\right)I(\mathscr{F}^{-1}_{q+4,m}(\Delta^2)\leq d_1)]\cr
&&+2\bbeta'\bW(\bI_p-\bM_{\bK})\bS_{\bK}^{-1}\bS\bbeta\mathbb{E}[\left(1-d\mathscr{F}^{-1}_{q+2,m}(\Delta^2)\right)I(\mathscr{F}^{-1}_{q+2,m}(\Delta^2)\leq d_1)]\cr
&&+{\rm tr}(\bC)\mathbb{E}[\left(1-d\mathscr{F}^{-1}_{q+4,m}(\Delta^2)\right)^2I(\mathscr{F}^{-1}_{q+4,m}(\Delta^2)\leq d_1)]\cr
&&+2d{\rm tr}(\bC)\mathbb{E}[\mathscr{F}^{-1}_{q+4,m}(\Delta^2)\left(1-d\mathscr{F}^{-1}_{q+4,m}(\Delta^2)\right)I(\mathscr{F}^{-1}_{q+4,m}(\Delta^2)\leq d_1)]\cr
&&+\bbeta'\bC\bbeta\mathbb{E}[\left(1-d\mathscr{F}^{-1}_{q+4,m}(\Delta^2)\right)^2I(\mathscr{F}^{-1}_{q+4,m}(\Delta^2)\leq d_1)]\cr
&&+2d\bbeta'\bC\bbeta\mathbb{E}[\mathscr{F}^{-1}_{q+4,m}(\Delta^2)\left(1-d\mathscr{F}^{-1}_{q+4,m}(\Delta^2)\right)I(\mathscr{F}^{-1}_{q+4,m}(\Delta^2)\leq d_1)]\cr
{\rm R}(\hbbeta_n^{\rm IPT})&=&{\rm R}(\hbbeta_n^{\rm S})\cr
&&+\sigma^2{\rm tr}\left(\bS_{\bK}^{-1}\bS\bS_{\bK}^{-1}\bW\right)+\bbeta'(\bS_{\bK}^{-1}\bS-\bI_p)\bW(\bS\bS_{\bK}^{-1}-\bI_p)\bbeta\cr
&&-2{\rm tr}(\bB)\mathbb{E}[\left(1-d\mathscr{F}^{-1}_{q+2,m}(\Delta^2)\right)I(\mathscr{F}^{-1}_{q+2,m}(\Delta^2)\leq \mathscr{F}_{q+2,m}(\alpha))]\cr
&&-2\bbeta'\bB\bbeta\mathbb{E}[\left(1-d\mathscr{F}^{-1}_{q+4,m}(\Delta^2)\right)I(\mathscr{F}^{-1}_{q+4,m}(\Delta^2)\leq \mathscr{F}_{q+4,m}(\alpha))]\cr
&&+2\bbeta'\bW(\bI_p-\bM_{\bK})\bS_{\bK}^{-1}\bS\bbeta\mathbb{E}[\left(1-d\mathscr{F}^{-1}_{q+2,m}(\Delta^2)\right)I(\mathscr{F}^{-1}_{q+2,m}(\Delta^2)\leq \mathscr{F}_{q+2,m}(\alpha))]\cr
&&+{\rm tr}(\bC)\mathbb{E}[\left(1-d\mathscr{F}^{-1}_{q+4,m}(\Delta^2)\right)^2I(\mathscr{F}^{-1}_{q+4,m}(\Delta^2)\leq \mathscr{F}_{q+4,m}(\alpha))]\cr
&&+2d{\rm tr}(\bC)\mathbb{E}[\mathscr{F}^{-1}_{q+4,m}(\Delta^2)\left(1-d\mathscr{F}^{-1}_{q+4,m}(\Delta^2)\right)I(\mathscr{F}^{-1}_{q+4,m}(\Delta^2)\leq \mathscr{F}_{q+4,m}(\alpha))]\cr
&&+\bbeta'\bC\bbeta\mathbb{E}[\left(1-d\mathscr{F}^{-1}_{q+4,m}(\Delta^2)\right)^2I(\mathscr{F}^{-1}_{q+4,m}(\Delta^2)\leq \mathscr{F}_{q+4,m}(\alpha))]\cr
&&+2d\bbeta'\bC\bbeta\mathbb{E}[\mathscr{F}^{-1}_{q+4,m}(\Delta^2)\left(1-d\mathscr{F}^{-1}_{q+4,m}(\Delta^2)\right)I(\mathscr{F}^{-1}_{q+4,m}(\Delta^2)\leq \mathscr{F}_{q+4,m}(\alpha))]
\end{eqnarray*}
\end{theorem}
For the proof, see the Appendix.

\section{High Dimensional Sparse Regression Models}
\label{sec:HD}
Consider the model \eqref{full:model} when $p>n$. We assume that $\bx_i$ can be decomposed as $\bx_i=(\bx_{iA},\bx_{iB})$ with $\bx_{iA}=(\bx_{i1}\dots,\bx_{i(p-q)})'\in\mathbb{R}^{p-q}$ and $\bx_{iB}=(\bx_{i(p-q+1)}\dots,\bx_{ip})'\in\mathbb{R}^{q}$, where $p-q$ is smaller than the sample size $n$, and let $\bX_A=(\bx_{1A},\dots,\bx_{nA})'\in\mathbb{R}^{n\times (p-q)}$ and $\bX_B=(\bx_{1B},\dots,\bx_{nB})'\in\mathbb{R}^{n\times q}$ be
the matrices associated with the $\bx_{iA}$'s and $\bx_{iB}$'s, respectively. Finally, let $\bX=(\bX_A,\bX_B)'\in\mathbb{R}^{n\times p}$ be the matrix including all predictive variables. Then model \eqref{full:model} can be re-expressed as follows:
\begin{equation}\label{ABl:model}
\bY=\bX_{A} \bm\beta_A+\bX_B \bm\beta_B +\bm \varepsilon,
\end{equation}
where $\bm\beta=(\bm\beta'_A,\bm\beta'_B)'\in\mathbb{R}^{p}$, $\bm\beta'_A\in\mathbb{R}^{p-q}$ and $\bm\beta'_B\in\mathbb{R}^{q}$ are unknown regression coefficient vectors.

We are essentially interested in the estimation of $\bm\beta_A$ when it is plausible that $\bm\beta_B$ is a set of nuisance co-variates. This situation may arise when there is over-modelling and
one wishes to cut down the irrelevant part from the model \eqref{ABl:model}. 


In real-life applications, $\bX_{A}$ often contains a small set of relevant predictors via prior knowledge
or preliminary analysis while $\bX_{B}$ collects a large number of predictors, whose
statistical significance is still not clear and thus needs to be investigated. So, one need to test the following statistical hypotheses:
\begin{equation}
\label{null_test}
H_0:\bm\beta_B= \bm0 \text{ versus } H_1:\bm\beta_B \neq \bm0.
\end{equation}

In order to test the null hypothesis in \eqref{null_test}, we use the following the test statistic
\begin{equation}
\label{HD_test_stat}
\mathscr{T}_n = T_2/\left \{2\sigma^4\text{tr}(\Sigma^2_{B|A})^{1/2}  \right \},
\end{equation}
where $T_2=T_1-\hsigma_{\rm LSE}^2n^{-1}q\text{tr}(\bM\bQ)$, $\hsigma_{\rm LSE}^2=q^{-1}(\bY-\bX_A\hbbeta_A^{\rm LSE})'(\bY-\bX_A\hbbeta_A^{\rm LSE})$,  $\bQ=(\bI_n-\bH_A)$, $\bH_A = \bX_A(\bX'_A\bX_A)^{-1}\bX'_A$, $T_1=n^{-1}\sum_{j\in S}\tilde{\bY'}\tilde{\bX_j}\tilde{\bX'_j}\tilde{\bY}$, (where $S=\left\{(p-q+1),\dots,p\right\}$, $\tilde{\bY} = (\bI_p-\bH_A)\bY$ and $\tilde{\bX_j} = (\bI_p-\bH_A)\bX_j$) and $\bM=q^{-1}\sum_{j \in S}\tilde{\bX_j}\tilde{\bX'_j}$. 

\begin{theorem}
\label{HD_test_teo}
Under the conditions (C1) and (C2) which are given in \citet{Lan-et-al2014}, it is obtained that $\mathscr{T}_n\overset{\mathcal D}{\to} Z$, $Z\sim N(0,1)$.
\end{theorem}
\begin{proof}
The proof and further information can be found in \citet{Lan-et-al2014}.
\end{proof}

Thanks to this theorem, we can test a subset of
regression coefficients in HD case. Accordingly, Theorem \ref{HD_test_teo} indicates that, for a given significance level $\alpha$, we reject the null hypothesis in \eqref{null_test} if $\mathscr{T}_n>z_{1-\alpha}$, where $z_{\alpha}$ stands for the $\alpha$th quantile of a standard normal distribution. 

\subsection{Proposed estimators for HD case}
In this case, the proposed estimators remain the same, except the test statistic and the pre-specified matrix $\bH$. For our purpose, we use the test statistic in Theorem \ref{HD_test_teo} and take ${\bH}= [\bm0_{q \times (p-q)},\bI_{q \times q}]$. 

As an instance the high-dimensional PTE is defined as 
\begin{equation}
\label{beta_HDPT}
{\hbbeta}_n^{\textrm{HD-PT}}=\hbbeta_n^{\textrm{GRR}}-\left (\hbbeta_n^{\textrm{GRR}}-\hbbeta_n^{\textrm{RGRR}}  \right ) \textrm{I}\left(\mathscr{T}_{n}\leq  z_{1-\alpha}\right).
\end{equation}%
We also have
\begin{equation*}
\bm{\widehat{\beta}}_n^{\textrm{HD-S}}=\bm{\widehat{\beta}}_n^{\textrm{GRR}}-d^\ast\left( \bm{\widehat{\beta}}_n^{\textrm{GRR}}-\bm{\widehat{\beta}}_n^{\textrm{RGRR}}\right)\mathscr{T}_{n}^{-1},
\end{equation*}
where the shrinkage parameter $d^\ast$ is obtained by cross validation which minimizes the mean square error of $\bm{\widehat{\beta}}_n^{\textrm{HD-S}}$.

\subsection{Properties: High-dimension}
In this section, we provide the limiting distribution of the proposed shrinkage estimators. Before, we give upper bounds of GRR and RGRR estimators. To this end, let 
\begin{eqnarray}
F(\bbeta)&=&(\bY-\bX\bbeta)'(\bY-\bX\bbeta)+\bbeta'\boldsymbol{K}\bbeta\cr
\end{eqnarray}
Apparently, $\hbbeta_n^{GRR}=\argmin_{\bbeta\in\mathbb{R}^p} F(\bbeta)$ 
Now, consider that
\begin{eqnarray}
(\hbbeta_n^{GRR})'\boldsymbol{K}\hbbeta_n^{GRR}&\leq& (\bY-\bX\hbbeta_n^{GRR})'(\bY-\bX\hbbeta_n^{GRR})
+(\hbbeta_n^{GRR})'\boldsymbol{K}\hbbeta_n^{GRR}\cr
&=&F(\hbbeta_n^{GRR})\leq F(\boldsymbol{0})=\bY'\bY
\end{eqnarray}
Let $k_{(1)}=\min_{1\leq i\leq p}(k_i)$ be known. Then, we obtain
\begin{eqnarray}\label{eq48}
(\hbbeta_n^{\rm GRR})'\hbbeta_n^{\rm GRR}&\leq& k_{(1)}^{-1}\bY'\bY
\quad \Rightarrow \quad \mathbb{E}((\hbbeta_n^{GRR})'\hbbeta_n^{GRR})\leq k_{(1)}^{-1}\mathbb{E}(\bY'\bY)
\end{eqnarray}
Using the Courant Theorem, we have that 
\begin{eqnarray}\label{eq49}
\frac{(\hbbeta_n^{\rm RGRR})'(\hbbeta_n^{\rm RGRR})'}{(\hbbeta_n^{\rm GRR})'\hbbeta_n^{\rm GRR}}\leq \lambda_{\max}(\bM_{\rm \bK}^2)=\lambda_{\max}^2(\bM_{\rm \bK}).
\end{eqnarray}
Therefore, by \eqref{eq48} together with \eqref{eq49}, $\mathbb{E}((\hbbeta_n^{\rm RGRR})'\hbbeta_n^{RGRR})\leq k_{(1)}^{-1}\lambda_{\max}^2(\bM_{\rm \bK})\mathbb{E}(\bY'\bY)$.

Now, define the high-dimensional shrinkage (HDS) estimator as
\begin{eqnarray}
\hbbeta_n^{\rm HD-Shrinkage}=\hbbeta_n^{\rm GRR}-\left(\hbbeta_n^{\rm GRR}-(\hbbeta_n^{\rm RGRR})'\right)g(\mathscr{T}_n)
\end{eqnarray}
Then, $\hbbeta_n^{\rm HD-Shrinkage}=\bA_n\hbbeta_n^{\rm GRR}$, where $\bA_n=\bI_p-(\bI_p-\bM_{\rm \bK})g(\mathscr{T}_n)$. Using Theorem \ref{HD_test_teo}, $\bA_n\overset{\mathcal D}{\to}\bI_p-(\bI_p-\bM_{\rm \bK})g(Z)$ for continuous functions $g(\cdot)$. On the other hand, $\hbbeta_n^{\rm GRR}\overset{\mathcal P}{\to}\bS_{\rm \bK}^{-1}\bS\bbeta$. Thereofore, by Slutsky's theorem, $\hbbeta_n^{\rm HD-Shrinkage}\overset{\mathcal D}{\to}\bS_{\rm \bK}^{-1}\bS\bbeta-(\bI_p-\bM_{\rm \bK})\bS_{\rm \bK}^{-1}\bS\bbeta g(Z)$. However, not all functions $g(\cdot)$ are continuous. 

Using the result of \citet{saleh2006}, we have the following theorem for the high-dimensional case, without proof. 
\begin{theorem} Under the definitions at above, we have
\begin{eqnarray}
\hbbeta_n^{\rm HD-PT}&\overset{\mathcal D}{\to}&[\bI_p-(\bI_p-\bM_{\rm \bK})I(Z\leq z_\alpha)]\bS_{\rm \bK}^{-1}\bS\bbeta\cr
\hbbeta_n^{\rm HD-SPT}&\overset{\mathcal D}{\to}&[\bI_p-\omega(\bI_p-\bM_{\rm \bK})I(Z\leq z_\alpha)]\bS_{\rm \bK}^{-1}\bS\bbeta\cr
\hbbeta_n^{\rm HD-S}&\overset{\mathcal D}{\to}&[\bI_p-d(\bI_p-\bM_{\rm \bK})Z^{-1}]\bS_{\rm \bK}^{-1}\bS\bbeta\cr
\hbbeta_n^{\rm HD-PS}&\overset{\mathcal D}{\to}&[(\bI_p-\bM_{\rm \bK})+(\bI_p-\bM_{\rm \bK})(1-dZ^{-1})I(Z>d)]\bS_{\rm \bK}^{-1}\bS\bbeta\cr
\hbbeta_n^{\rm HD-IPT}&\overset{\mathcal D}{\to}&[\bM_{\rm \bK} dZ^{-1}+(1-dZ^{-1})I(Z>d)[1-I(Z\leq z_\alpha)(\bI_p-\bM_{\rm \bK})]
]\bS_{\rm \bK}^{-1}\bS\bbeta\cr
\end{eqnarray}
\end{theorem}

\section{Penalized Estimation}
\label{sec:PE}
In this paper, we only consider the penalized least square regression methods to obtain estimators for the model parameters in a multiple regression models when $p<n$ and $p>n$. The key idea in penalized regression methods is minimizing an objection function $L_{\lambda}$ in the form of
\begin{equation}
\label{objective}
L_{\lambda}(\bbeta)=(\by -\bX\bbeta)'(\by -\bX\bbeta) + \lambda
P(\bm \beta)
\end{equation}
to obtain the estimates of the parameter. The first term in the objective function is the sum of the squared error loss, the second term $\rho$ is a penalty function, and $\lambda$ is a tuning parameter which controls the trade-off between two components of $L_{\lambda}$.

The penalty function is usually chosen as a norm on $\mathbb{R}^p$,
\begin{equation}\label{eq:general:penalizedLS}
P(\bm \beta)=\sum_{j=1}^p|\beta_j|^{\gamma}, \ {\gamma}> 0.
\end{equation}
This class of estimator are called the bridge estimators, proposed by \citet{Frank1993}.

The ridge regression (\citet{Ho-Ke1970}, \citet{Frank1993}) minimizes the residual sum of squares subject to an $\mathit{l}_2$-penalty, that is,
\begin{equation}\label{eq:ridge:solution}
 \hbbeta_n^{\textrm{Ridge}} = \mbox{arg}\min_{\bm \beta}\left\{\sum_{i=1}^n (y_i- \sum_{j=1}^{p}x_{ij}\beta_j)^2 + \lambda
\sum_{j=1}^{p}\beta_j^2 \right\},
\end{equation}
where $\lambda$ is a tuning parameter.

\subsection{LASSO} The LASSO was proposed by \citet{Tibshirani1996}, which performs variable selection and parameter estimation simultaneously thanks to the $\mathit{l}_1$-penalty.
The LASSO estimates are defined by
\begin{equation}
\label{eq:LASSO:solution}
  \hbbeta_n^{\textrm{LASSO}} = \argmin_{\bbeta}\left\{\sum_{i=1}^n (y_i
- \sum_{j=1}^{p}x_{ij}\beta_j )^2 + \lambda
\sum_{j=1}^{p}|\beta_j| \right\}.
\end{equation}

\subsection{ALASSO}

\citet{Zou2006} introduced the ALASSO by modifying the LASSO penalty by using adaptive weights on $\mathit{l}_1$-penalty with the regression coefficients. 

The ALASSO $\widehat\beta^{\textrm{ALASSO}}$ are obtained by
\begin{equation}\label{eq:adALASSO:ch4}
 \hbbeta_n^{\textrm{ALASSO}} = \argmin_{\bbeta} \left\{
\sum_{i=1}^{n} (y_i -\sum_{j=1}^{p}x_{ij}\beta_j)^2 +
\lambda \sum_{j=1}^{p} \widehat{w}_j |\beta_j| \right\},
\end{equation}
where the weight function is
\[
 \widehat{w}_j = \frac{1}{|\widehat{\beta}^*_j|^\gamma}; \quad \gamma>0,
\]
and $\widehat{\beta}_j^*$ is a root-n-consistent estimator of $\beta$.
The minimization procedure for ALASSO solution does not induce any computational difficulty and can be solved very efficiently, for the details see section 3.5 in \citet{Zou2006}.

\subsection{SCAD}
Although the LASSO method does both shrinkage and variable
selection due to the nature of the $\mathit{l}_1$-penalty by setting many coefficients identically to zero, it does
not possess oracle properties, as discussed in
\citet{FanLi2001}. To overcome the inefficiency of
traditional variable selection procedures, they
proposed SCAD\index{SCAD} to select variables and estimate the
coefficients of variables automatically and simultaneously. Given the tuning parameters $a >2$ and $\lambda>0$, the SCAD penalty at $\beta$ is 
\begin{equation*}
\label{eq:scad:solution}
J_\lambda (\beta;a )=\left\{\begin{matrix}
\lambda \left | \beta \right |, &\left | \beta \right |\leq \lambda  \\ 
 -\left (\beta^2-2a \lambda \left | \lambda  \right |+\lambda ^2  \right )/\left [ 2(a -1) \right ],&\lambda < \left | \beta \right |\leq a\lambda  \\ 
 (a+1)\lambda ^2/2& \left | \beta \right | >a\lambda.
\end{matrix}\right.
\end{equation*}%
Hence, the SCAD estimation is given by
\begin{equation*}
\hbbeta_n^{\rm SCAD}\underset{\bbeta} =\argmin_{\bbeta}\left\{\sum_{i=1}^n  \left
   (y_i-\sum_{j=1}^{p}x_{ij}\beta_j \right)^2 +
\sum_{j=1}^{p}J_\lambda (\beta _{j};a )\right\}.
\end{equation*}%

\subsection{MCP}

\citet{Zhang2007} introduced a new penalization method for variable 
selection, which is given by
 \[
  \hbbeta_n^{{\rm MCP}} = \argmin_{\bbeta}\left\{\sum_{i=1}^n  \left
   (y_i-\sum_{j=1}^{p}x_{ij}\beta_j \right)^2+ \sum_{j=1}^{p}
 \rho(|\beta_j|;\lambda)\right\},
\]
where $\rho(\cdot;\lambda)$ is the MCP penalty given by
\[\rho(t;\lambda)=\lambda \int_0^t (1-\frac{x}{\gamma \lambda})^+ \,dx\]
where $\gamma>0$ and $\lambda$ are regularization and penalty parameters respectively.

\subsection{Elastic-Net}
The Elastic-Net was proposed by \citet{ZouHastie2005} to overcome the limitations of the LASSO and Ridge methods. 

\begin{equation}\label{eq:ENET:solution}
  \hbbeta^{\textrm{ENET}} = \mbox{arg}\min_{\beta}\left\{\sum_{i=1}^n (y_i - \sum_{j=1}^{d}x_{ij}\beta_j )^2 + \lambda_1
\sum_{j=1}^{d}|\beta_j| +\lambda_2
\sum_{j=1}^{d}\beta_j^2 \right\},
\end{equation}
where $\lambda_2$ is the ridge penalty parameter, penalizing the sum of the squared regression coefficients and $\lambda_1$ is the LASSO penalty, penalizing the sum of the absolute values of the regression coefficients, respectively.

\subsection{MNET}

\citet{Huangetal2010} introduced the MNET estimation which use MCP penalty term instead of the $\mathit{l}_1$-penalty in \eqref{eq:ENET:solution}.

\begin{equation}\label{eq:MNET:solution}
  \hbbeta^{\textrm{MNET}} = \mbox{arg}\min_{\beta}\left\{\sum_{i=1}^n (y_i - \sum_{j=1}^{d}x_{ij}\beta_j )^2 + \sum_{j=1}^{p}
 \rho(|\beta_j|;\lambda_1) +\lambda_2
\sum_{j=1}^{d}\beta_j^2 \right\}.
\end{equation}

Similar to the elastic net of \citet{ZouHastie2005}, the MNET also tends to select or drop highly correlated predictors together. However, unlike the elastic net, the MNET is selection consistent and equal to the oracle ridge estimator with high probability under reasonable conditions

\section{Simulation Study}
\label{sec:SS}

We conduct Monte-Carlo simulation experiments to study the performances of the proposed estimators under various practical settings. In order to generate the response variables, we consider
\begin{equation*}
\bY=\bX \bm\beta+ \sigma\bm \varepsilon,
\end{equation*}
where $\bm \varepsilon\sim \mathscr{N}(0,\boldsymbol{I}_p)$. 

\subsection{Risk Performances}

We consider the regression coefficients are set
\linebreak $\bbeta=\left( \bbeta_{1}',\boldsymbol{%
\beta }_{2}'\right)' =\left( \bm1'_{p-q},\bm{0}_{q}'\right)'$, where $\bm1_{p-q}$ and $\bm{0}_q$ mean the vectors of 1 and 0 with dimensions $p-q$ and $q$, respectively. In order to investigate the behaviour of the estimators, we define $\Delta^{\ast}=\left\Vert \boldsymbol{\beta -\beta }_{0}\right\Vert $, where $\bbeta_{0}=\left( \bm1'_{p-q},\boldsymbol{0}_{q}'%
\right)'$ and $\left\Vert \cdot \right\Vert $ is the Euclidean norm. We considered $\Delta^{\ast}$ values between $0$ and $4$. If $\Delta^{\ast}=0$, then it means that we
will have $\bbeta=\left( \bm1_{p-q}',\bm{0}_{q}'\right)'$ to generated the response while we will have $\bbeta=(\bm1'_{p-q},0.5, \bm{0}_{q-1}')'$ when $\Delta^{\ast}>0$, say $\Delta^{\ast}=0.5$. When we increase the number of $\Delta^{\ast}$, it indicates the degree of violation of null hypothesis. We also set $\bX \sim \mathscr{N}(0,\boldsymbol{\Sigma})$, where $\Sigma_{kj}=\rho^{|k-j|}$, $k=1,\dots,p$, $j=1,\dots,p$.



The performance of one of the suggested estimator was evaluated by
using MSE criterion. Also, the relative mean square efficiency
(RMSE) of the $\hbbeta_n^{\ast}$ to the benchmark estimator (BE)
 $\hbbeta_n^{\rm BE}$ is indicated by
\begin{equation*}
\textnormal{RMSE}\left( \hbbeta_n^{\rm BE}:\hbbeta_n^{\ast }\right) =\frac{\textnormal{MSE}\left( \hbbeta_n^{\rm BE}\right) }{\textnormal{MSE}\left( \hbbeta_n^{\ast}\right) },
\end{equation*}
where $\hbbeta_n^{\rm BE}$ will be $\hbbeta_n^{\rm GRR}$ or $\hbbeta_n^{\rm HD-GRR}$ and  $\hbbeta_n^{\ast }$ is one of the listed estimators. If the RMSE of an estimator is larger than one, it indicates to superior to $\hbbeta_n^{\rm BE}$.





\begin{figure}
    \centering
    \begin{subfigure}[b]{1\textwidth}
        \centering
        \includegraphics[width=14cm,height=10cm]{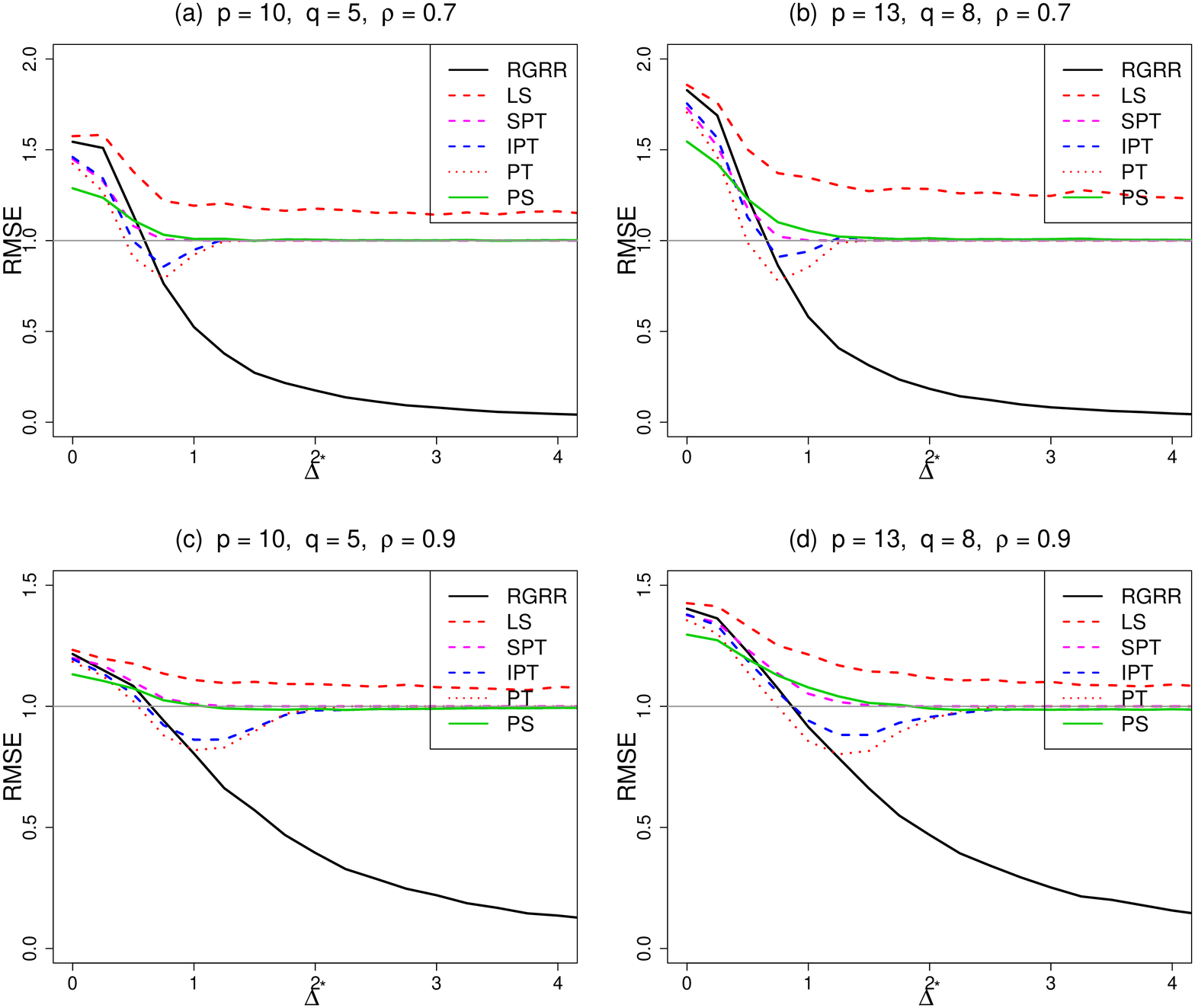}
        \caption{LD case}
        \label{fig:LD}
    \end{subfigure}
    \\
    \begin{subfigure}[b]{1\textwidth}
        \centering
        \includegraphics[width=14cm,height=10cm]{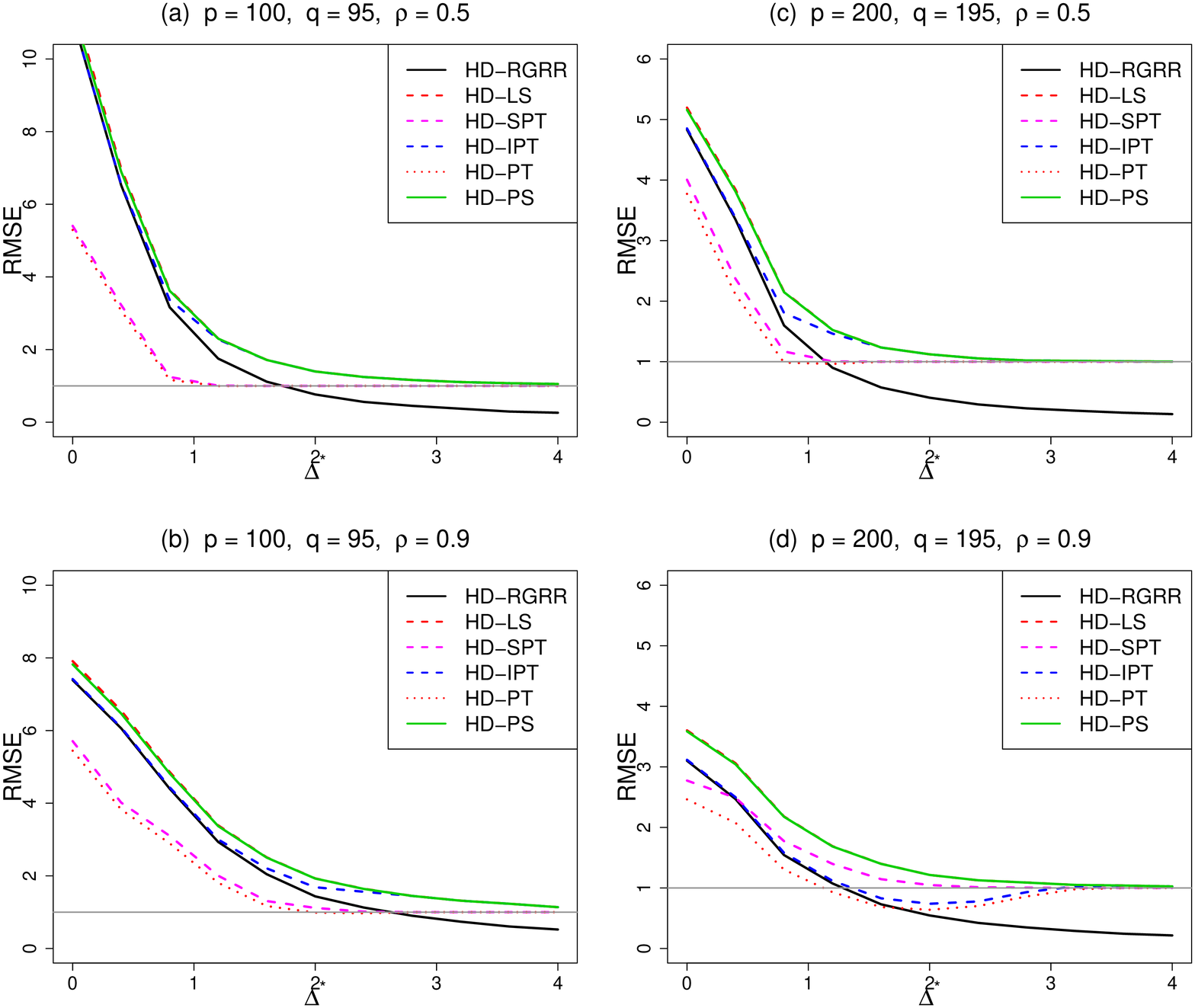}
        \caption{HD case}
        \label{fig:HD}
    \end{subfigure}
    \caption{RMSE of the estimators as a function of the
non-centrality parameter $\Delta^{\ast}$ when $n=50$}
    \label{Deltas}
\end{figure}

The results are shown in Figures \ref{fig:LD} and \ref{fig:HD} is for LD and HD cases, respectively. The findings can be summarized as follows:

\begin{enumerate} 
\setlength{\itemsep}{0pt} \setlength{\parskip}{0pt}

   \item [\bf{LD case:}] 
   In general, when $\Delta^{\ast}=0$, RGRR outperforms all listed estimators, except that LS because we select $\omega$ tuning parameter of LS with CV. So, the performance of LS may be superior to RGRR. In contrast, after the small interval near $\Delta^{\ast}$, the RMSE of the RGRR decreases and goes to zero while the RMSE of LS, SPT and PS decrease and goes to one as $\Delta^{\ast}$ is large. On the other hand, the performance of Preliminary-types estimations, that are SPT, IPT and PT, outperform the PS as $\Delta^{\ast}$ is zero. When we increase the value of $\Delta^{\ast}$, the RMSEs of IPT and PT decrease, and they remain below one for intermediate values of  $\Delta^{\ast}$, and finally increase and approach to one for larger values of $\Delta^{\ast}$.

\item [\bf{HD case:}] We get similar pattern with LD case with some exceptions. Cleary, the RMSE of HD-LS and HD-PS is the best when the null hypothesis is true since we select their tuning parameters with CV. 
  
\end{enumerate}

In the following section, we compare the performance of listed estimations with LSE and penalty estimation techniques.
\subsection{Some Comparative Studies}
\label{SCS}


In each example of the LD cases, the simulated data contains a training dataset, validation data and an independent test set. We fitted the models only using the training data and the tuning parameters were selected using the validation data. In simulations, 
both we standardized the predictors to have zero mean and unit standard deviation before fitting the model and we center the response variable in order to omit the intercept term.
Finally, we computed two measures of performance, the test error (mean squared error) $\rm MSE_y = \frac{1}{n_{test}} \textbf{r}_{sim}'\textbf{r}_{sim}$ where $\textbf{r}_{sim} = \bx_i\bbeta-\bx_i\hbbeta$ and the mean squared error of the estimation of $\bbeta$ such that $\rm MSE_{\beta} = |\hbbeta-\bbeta|^2$ (see \cite{tutz:ulbricht2009}). We use the notation $\cdot/\cdot/\cdot$ to describe the number of observations in the training, validation and test set respectively, e.g. $100/100/200$. In each example of the HD cases, we did not split data as training, validation and test set. We just scaled data as the LD case. For both cases, we consider the $\rho=0.5,0.9$.

\begin{itemize}
\item[LD 1-] Each data set consists of $20/20/200$ observations. $\bbeta$ is specified by \linebreak $\bbeta' = (3, 1.5, 0, 0, 2, 0, 0, 0)$ and $\sigma = 3$. Also, we consider $\bX \sim \mathscr{N}(0,\boldsymbol{\Sigma})$, where $\Sigma_{kj}=\rho^{|k-j|}$.

\item[LD 2-] Each data set consists of $100/100/400$ observations and $40$ predictors, where $\beta_j=0$ when $j=1,\dots,10, 21,\dots,30$ and $\beta_j=2$ when $j=11,\dots,20, 31,\dots,40$. Also, we set $\sigma=1$, $\Sigma_{kj}=\rho^{|k-j|}$. 

\item[LD 3-]  We simulated 250 data sets consisting of 50/50/200 observations and 30 predictors.
We chose
\begin{equation*}
\bbeta' =\left( \underbrace { 2,\dots,2}_{5},\underbrace {0,\dots,0}_{25} \right). 
\end{equation*}
Also, we consider $\sigma = 9$ and $\bX \sim \mathscr{N}(0,\boldsymbol{\Sigma})$, where $\Sigma_{kj}=\rho^{|k-j|}$.

\item[HD 1-] We simulated 50 data sets consisting of 50 observations and 100 predictors, where $\bbeta' = (\bm1_5, \bm0_{p-5})$. Also, we set $\sigma=3$, $\Sigma_{kj}=\rho^{|k-j|}$.

\item[HD 2-] We simulated 50 data sets consisting of 50 observations and 150 predictors, where $\bbeta' = (\bm3_{10}, \bm0_{p-10})$. Also, we set $\sigma=2$, $\Sigma_{kj}=\rho^{|k-j|}$.

\item[HD 3-] We simulated 50 data sets consisting of 50 observations and 120 predictors $\left\{\bX_1,\dots,\bX_p\right\}$. If the predictors $\bX_t,\dots,\bX_{t+g}$ are linked together on the network, they would be categorized into a group, denoted as $\underbrace{\bX_t,\dots,\bX_{t+g}}_{\rm group}$ $(0<t<t+g<p)$. In this example, the design matrix is \linebreak $\bX=\left\{\underbrace{\bX_1\dots,\bX_{10}}_{\rm group\text{ }1},\dots,\underbrace{\bX_{31}\dots,\bX_{40}}_{\rm group\text{ }4},\bX_{41},\dots,\bX_{120} \right\}$ where $\bX \sim \mathscr{N}(0,\boldsymbol{\Sigma})$. The diagonal elements of $\boldsymbol{\Sigma}$ are 1, and the non-diagonal elements within the group $m(m=1,\dots,4)$ are $\rho$, and the non-diagonal
elements between any two groups are $0.1$. We also set $\bbeta' = (\bm5_{10},\bm{-5}_{10},\bm3_{10},\bm{-3}_{10},\bm0_{p-40})$ and $\sigma=3$.

\end{itemize}


\begin{table}[!htbp]
\centering
  \caption{Performances of methods for the simulated examples for LD cases}
  \label{table:LD:ex}
\begin{adjustbox}{width=.75\textwidth}
\begin{tabular}{lrrrrrrrrr}
\\[-1.8ex]\hline
\hline \\[-1.8ex]
&\multicolumn{2}{c}{LD 1} 
&\multicolumn{2}{c}{LD 2}
&\multicolumn{2}{c}{LD 3} \\

\cmidrule(lr){2-3} 
\cmidrule(lr){4-5} 
\cmidrule(lr){6-7} 

$\rho=0.5$&$\rm MSE_y$ &$\rm MSE_{\beta}$ 
&$\rm MSE_y$ &$\rm MSE_{\beta}$ 
&$\rm MSE_y$ &$\rm MSE_{\beta}$ 
\\
\\[-1.8ex]\hline
\hline \\[-1.8ex]
  GRR & 4.751 & 7.499 & 1.075 & 1.166 & 47.343 & 69.961 \\ 
  RGRR & 2.340 & 3.093 & 0.876 & 0.866 & 5.724 & 7.424 \\ 
  LS & 2.304 & 3.014 & 0.875 & 0.861 & 5.721 & 7.410 \\ 
  PT & 3.700 & 5.402 & 0.922 & 0.937 & 17.293 & 25.112 \\ 
  SPT & 3.683 & 5.359 & 0.921 & 0.932 & 17.294 & 25.102 \\ 
  PS & 3.717 & 5.580 & 0.899 & 0.900 & 10.555 & 14.893 \\ 
  IPT & 3.346 & 4.794 & 0.890 & 0.887 & 9.159 & 12.721 \\ 
  \cmidrule(lr){2-7}
  LSE & 6.432 & 10.779 & 1.282 & 1.560 & 121.944 & 199.726 \\ 
  Ridge & 3.820 & 4.555 & 1.580 & 1.396 & 8.009 & 5.305 \\ 
  LASSO & 3.652 & 4.511 & 1.118 & 1.080 & 8.786 & 7.145 \\ 
  ALASSO & 4.391 & 5.157 & 0.990 & 0.955 & 10.120 & 9.462 \\ 
  SCAD & 4.591 & 5.470 & 0.877 & 0.888 & 8.797 & 7.198 \\ 
  MCP & 4.964 & 5.924 & 0.876 & 0.885 & 9.518 & 8.386 \\ 
  ENET & 3.291 & 3.992 & 1.121 & 1.076 & 7.716 & 5.923 \\ 
  MNET & 2.802 & 3.679 & 0.881 & 0.864 & 7.973 & 6.725 \\

  \\[-1.8ex]\hline
\hline \\[-1.8ex]
&\multicolumn{2}{c}{LD 1} 
&\multicolumn{2}{c}{LD 2}
&\multicolumn{2}{c}{LD 3} \\

\cmidrule(lr){2-3} 
\cmidrule(lr){4-5} 
\cmidrule(lr){6-7} 

$\rho=0.9$&$\rm MSE_y$ &$\rm MSE_{\beta}$ 
&$\rm MSE_y$ &$\rm MSE_{\beta}$ 
&$\rm MSE_y$ &$\rm MSE_{\beta}$ 
\\
\\[-1.8ex]\hline
\hline \\[-1.8ex]

  GRR & 4.709 & 27.042 & 1.849 & 4.494 & 45.698 & 340.141 \\ 
  RGRR & 2.689 & 7.233 & 1.642 & 2.684 & 5.496 & 25.925 \\ 
  LS & 2.635 & 7.020 & 1.641 & 2.663 & 5.485 & 25.847 \\ 
  PT & 3.271 & 13.799 & 1.707 & 3.221 & 16.964 & 119.140 \\ 
  SPT & 3.222 & 13.606 & 1.705 & 3.203 & 16.956 & 119.080 \\ 
  PS & 3.372 & 15.046 & 1.669 & 2.922 & 10.045 & 63.851 \\ 
  IPT & 3.071 & 11.766 & 1.662 & 2.849 & 8.764 & 53.361 \\ 
  \cmidrule(lr){2-7}
  LSE & 7.683 & 60.845 & 2.027 & 6.537 & 120.168 & 1110.142 \\ 
  Ridge & 2.861 & 6.976 & 2.062 & 2.667 & 7.635 & 7.539 \\ 
  LASSO & 2.857 & 10.122 & 1.953 & 3.559 & 7.409 & 13.160 \\ 
  ALASSO & 3.462 & 13.013 & 1.895 & 5.425 & 8.210 & 19.630 \\ 
  SCAD & 4.007 & 15.929 & 1.724 & 3.437 & 7.564 & 15.995 \\ 
  MCP & 4.151 & 18.373 & 1.741 & 3.626 & 8.038 & 18.534 \\ 
  ENET & 2.586 & 7.594 & 1.899 & 2.331 & 6.730 & 9.848 \\ 
  MNET & 2.629 & 8.680 & 1.718 & 1.771 & 7.307 & 16.115 \\ 

\hline
\end{tabular}
\end{adjustbox}
\end{table}

In Table \ref{table:LD:ex}, we report the results of the examples by simulating $250$ data sets. The findings can be summarized as follows: For LD 1 case, the performance of LS is the best. The MNET has the best performance among penalized and LSE estimators in the sense of both measures when $\rho=0.5$. On the other hand, the ENET is the best performance for first measure and the Ridge is the best performance for second measure. 

\begin{table}[!htbp]
\centering
  \caption{Performances of methods for the simulated examples for HD cases}
  \label{table:HD:ex}
\begin{adjustbox}{width=.75\textwidth}
\begin{tabular}{lrrrrrrrrr}
\\[-1.8ex]\hline
\hline \\[-1.8ex]
&\multicolumn{2}{c}{HD 1} 
&\multicolumn{2}{c}{HD 2}
&\multicolumn{2}{c}{HD 3} \\

\cmidrule(lr){2-3} 
\cmidrule(lr){4-5} 
\cmidrule(lr){6-7} 

$\rho=0.5$&$\rm MSE_y$ &$\rm MSE_{\beta}$ 
&$\rm MSE_y$ &$\rm MSE_{\beta}$ 
&$\rm MSE_y$ &$\rm MSE_{\beta}$ 
\\
\\[-1.8ex]\hline
\hline \\[-1.8ex]
  GRR & 4.195 & 5.235 & 2.925 & 7.022 & 8.430 & 134.643 \\ 
  RGRR & 1.223 & 1.413 & 0.807 & 1.732 & 7.857 & 95.601 \\ 
  LS & 1.220 & 1.408 & 0.823 & 1.690 & 7.400 & 76.692 \\ 
  PT & 1.305 & 1.488 & 1.117 & 2.398 & 8.430 & 134.643 \\ 
  SPT & 1.302 & 1.484 & 1.130 & 2.363 & 8.430 & 134.643 \\ 
  PS & 1.220 & 1.409 & 0.823 & 1.691 & 7.401 & 76.704 \\ 
  IPT & 1.223 & 1.413 & 0.811 & 1.724 & 7.403 & 76.703 \\ 
  \cmidrule(lr){2-7}
  Ridge & 8.593 & 12.048 & 3.843 & 38.107 & 8.622 & 136.283 \\ 
  LASSO & 7.792 & 14.300 & 2.689 & 5.046 & 10.657 & 365.165 \\ 
  ALASSO & 6.915 & 14.355 & 1.179 & 2.960 & 40.906 & 769.563 \\ 
  SCAD & 8.274 & 31.417 & 1.308 & 5.256 & 83.896 & 1968.911 \\ 
  MCP & 8.406 & 32.662 & 1.563 & 7.340 & 62.658 & 1996.700 \\ 
  ENET & 7.791 & 14.281 & 2.693 & 5.047 & 10.478 & 326.043 \\ 
  MNET & 8.218 & 27.518 & 1.164 & 2.822 & 49.038 & 746.071 \\

  \\[-1.8ex]\hline
\hline \\[-1.8ex]
&\multicolumn{2}{c}{HD 1} 
&\multicolumn{2}{c}{HD 2}
&\multicolumn{2}{c}{HD 3} \\

\cmidrule(lr){2-3} 
\cmidrule(lr){4-5} 
\cmidrule(lr){6-7} 

$\rho=0.9$&$\rm MSE_y$ &$\rm MSE_{\beta}$ 
&$\rm MSE_y$ &$\rm MSE_{\beta}$ 
&$\rm MSE_y$ &$\rm MSE_{\beta}$ 
\\
\\[-1.8ex]\hline
\hline \\[-1.8ex]
  
  GRR & 4.311 & 18.157 & 2.156 & 10.813 & 8.325 & 185.210 \\ 
  RGRR & 0.672 & 3.627 & 0.640 & 7.000 & 6.240 & 54.178 \\ 
  LS & 0.676 & 3.600 & 0.766 & 6.471 & 8.215 & 55.966 \\ 
  PT & 1.120 & 4.847 & 0.911 & 7.496 & 8.325 & 56.210 \\ 
  SPT & 1.126 & 4.826 & 1.023 & 7.005 & 8.325 & 56.210 \\ 
  PS & 0.677 & 3.602 & 0.766 & 6.471 & 8.215 & 55.965 \\ 
  IPT & 0.671 & 3.622 & 0.655 & 6.960 & 8.215 & 55.965 \\ 
  \cmidrule(lr){2-7} 
  Ridge & 8.677 & 60.926 & 3.927 & 19.141 & 8.777 & 57.333 \\ 
  LASSO & 5.774 & 31.639 & 1.332 & 12.463 & 11.112 & 321.181 \\ 
  ALASSO & 4.167 & 24.054 & 3.163 & 52.811 & 46.857 & 1309.429 \\ 
  SCAD & 5.757 & 37.078 & 15.220 & 277.480 & 146.294 & 4051.477 \\ 
  MCP & 6.415 & 47.537 & 14.937 & 334.231 & 143.563 & 3344.942 \\ 
  ENET & 5.773 & 31.607 & 1.322 & 12.165 & 9.639 & 175.807 \\ 
  MNET & 6.090 & 44.426 & 2.894 & 19.139 & 57.465 & 465.671 \\ 

\hline
\end{tabular}
\end{adjustbox}
\end{table}

Also, the another proposed estimations are competitive for both $\rho$ values. For LD 2 case, the performances of SCAD and MCP are better while LS and the other shrinkage estimations are again competitive and more effective than ALASSO, LASSO, Ridge and ENET when $\rho=0.5$. In contrast, the performance of the suggested estimations are better than all others for first measure when $0.9$ while their performance may be less effective compare to second measure. For LD 3 case, the LS has minimum first measure for both $\rho$ values. 
ENET and MNET perform well. Overall, the proposed shrinkage estimators perform well in low-dimensional case. We turn the readers attention to the LD 2 and LD 3 cases. Although the SCAD and MCP perform the best among all, but the performance of shrinkage estimators is superior to LASSO, ALASSO and ENET. This is surprisingly interesting since the penalty estimators do variable selection and one may expect smaller error in prediction compared to shrinkage estimators. 


In Table \ref{table:HD:ex}, we report the results of the examples by simulating $50$ data sets. We did not give LSE here since it does not exist in case of HD. The findings can be summarized as follows: For HD 1 case, we consider five strong signals and 95 noises. The results show that the proposed estimations outshines the others not only for both measures but also for both $\rho$ values. For HD 2 case, we consider ten more strong signals and 140 noises. Similar to the results of HD 1, we again see that our suggested methods perform well. For HD 3 case, we consider group correlations among some predictors. Even though the Ridge, LASSO and ENET are competitive, our proposed estimation methods outshine.  Overall, it can be safety concluded that the GRR and RGRR perform better than the Ridge, and we suggest to use the listed estimations when HD case.

\section{Real Data Applications}
\label{sec:RDA}
In this section, we illustrated two real data examples in order to investigate the performance of the suggested estimations strategies.

\subsection{Pollution Data}
We first consider the Pollution data set which is analyzed by \citet{McSc1973}. This data includes $p=15$ measurements on mortality rate and explanatory variables, which are air-pollution, socio-economic and meteorological, for $n=60$ US cities in 1960. The data are freely available from Carnegie Mellon University's StatLib \texttt{(http://lib.stat.cmu.edu/datasets/)}. In Table \ref{Tab:variables:pollution}, we listed variables.

\begin{table}[!htbp]
\begin{adjustbox}{width=.9\textwidth}
\small
\centering
\begin{tabular}{ll}
\hline
\textbf{Variables} & \textbf{Descriptions} \\
\hline \hline

\textbf{Dependent Variable} &\\
mort    &   Total age-adjusted mortality rate per 100.000\\
\hline \hline

\textbf{Covariates}  & \\
prec    &   Average annual precipitation in inches\\
jant    &   Average January temperature in degrees F\\
jult    &   Average July temperature in degrees F\\
humid   &   Annual average \% relative humidity at 1pm\\
ovr65   &   \% of 1960 SMSA population aged 65 or older\\
popn    &   Average household size\\
educ    &   Median school years completed by those over 22\\
hous    &   \% of housing units which are sound \& with all facilities\\
dens    &   Population per sq. mile in urbanized areas, 1960\\
nonw    &   \% non-white population in urbanized areas, 1960\\
wwdrk   &   \% employed in white collar occupations\\
poor    &   \% of families with income $< 3000$\\
hc      &   Relative hydrocarbon pollution potential of hydrocarbons\\
nox     &   Relative hydrocarbon pollution potential of nitric oxides\\
so2     &   Relative hydrocarbon pollution potential of sulphur dioxides\\
\hline 
\end{tabular}
\end{adjustbox}
\caption{Lists and Descriptions of Variables for Pollution data set
\label{Tab:variables:pollution}}
\end{table}

Since the prior information is not available here, the constraint on the parameters is usually either obtained through expert opinion or obtained by using
existing variable selection techniques, such as AIC or BIC, among others. In this example we use the Best Subset Selection (BSS) with ``one standard error" rule. It 
showed that prec, jant, educ, nonw, so2 are the most important co-variates. Hence, we construct the shrinkage techniques by using both the full-model and the candidate sub-model.

Our results are based on $250$ case re-sampled bootstrap samples. Since there is no noticeable variation for larger number of replications, we did not consider further values. We first split the data in two parts which are train and test sets. Then, we fit the model on the train set and  calculate $\rm MSE_y$ and $\rm MSE_{\beta}$ and prediction error ($\rm PE = \mathbb{Y}_{test}-\mathbb{\widehat{Y}}_{test}$) by using test set based on 10-fold CV for each bootstrap replicate. Note that the predictors were first standardized to have zero mean and unit standard deviation, and the response variable is centered before fitting the model. The results are shown in Table \ref{tab:pol:res}.

\begin{table}[!htbp]
\begin{adjustbox}{width=1\textwidth}
\centering
\begin{tabular}{rcccccccc}
  \hline
 & prec & jant & educ & nonw & so2 & $\rm RMSE_y$ & $\rm RMSE_{\beta}$ & \rm RPE \\ 
  \hline
  GRR & 14.919(0.110) & -16.394(0.129) & -13.547(0.155) & 41.659(0.142) & 2.518(0.154) & 1.000 & 1.000 & 1.000 \\ 
  RGRR & 11.829(0.094) & -16.945(0.076) & -12.376(0.104) & 37.702(0.101) & 10.896(0.211) & 3.586 & 41.698 & 1.236 \\ 
  LS & 11.829(0.094) & -16.925(0.076) & -12.366(0.104) & 37.694(0.101) & 10.896(0.211) & 3.591 & 41.768 & 1.236 \\ 
  PT & 12.055(0.106) & -16.945(0.087) & -12.376(0.107) & 37.851(0.110) & 9.845(0.224) & 2.031 & 5.467 & 1.127 \\ 
  SPT & 12.055(0.106) & -16.930(0.087) & -12.364(0.107) & 37.851(0.109) & 9.845(0.224) & 2.032 & 5.468 & 1.127 \\ 
  PS & 13.004(0.093) & -16.695(0.090) & -12.806(0.120) & 38.987(0.107) & 6.577(0.179) & 2.132 & 3.988 & 1.155 \\ 
  IPT & 12.033(0.103) & -17.003(0.084) & -12.376(0.105) & 37.844(0.107) & 9.910(0.218) & 2.611 & 10.033 & 1.177 \\ 
  \cmidrule(lr){2-9} 
  LSE & 19.524(0.115) & -19.875(0.154) & -15.514(0.188) & 39.685(0.157) & 4.098(0.147) & 0.591 & 0.542 & 0.837 \\ 
  Ridge & 17.959(0.069) & -14.231(0.087) & -7.939(0.095) & 33.118(0.087) & 14.702(0.049) & 1.694 & 15.126 & 1.123 \\ 
  LASSO & 19.495(0.115) & -19.726(0.151) & -15.076(0.188) & 39.760(0.155) & 4.856(0.147) & 0.619 & 0.609 & 0.851 \\ 
  ALASSO & 19.345(0.117) & -19.656(0.150) & -15.305(0.185) & 39.902(0.155) & 4.320(0.156) & 0.614 & 0.582 & 0.849 \\ 
  SCAD & 19.647(0.124) & -19.891(0.161) & -15.665(0.194) & 39.716(0.165) & 4.263(0.150) & 0.519 & 0.530 & 0.786 \\ 
  MCP & 19.642(0.124) & -19.881(0.161) & -15.629(0.194) & 39.756(0.165) & 4.258(0.150) & 0.521 & 0.532 & 0.787 \\ 
  ENET & 19.511(0.114) & -19.758(0.151) & -15.095(0.189) & 39.750(0.155) & 4.827(0.147) & 0.618 & 0.607 & 0.850 \\ 
  MNET & 19.635(0.123) & -19.825(0.159) & -15.401(0.191) & 39.688(0.164) & 4.491(0.145) & 0.553 & 0.575 & 0.808 \\ 
   \hline
\end{tabular}
\end{adjustbox}
\caption{Estimate and standard error (in the parenthesis) for only significant coefficients for the pollution data. The last three column gives the relative $\rm MSE_y$, $\rm MSE_{\beta}$ and $\rm PE$, respectively, based on bootstrap simulation with respect to the GRR. If one of them is larger than one, then it is superior to the GRR.}
\label{tab:pol:res}
\end{table}

As expected, the performance of the RGRR is the best since the data is re-sampled from an empirical distribution where the candidate subspace is nearly true. Also, the proposed estimators are superior to the full model LSE and penalty estimators, which is in agreement with our theoretical and simulation results.

\subsection{Eye Data}

This data set contains gene expression data of mammalian eye tissue samples, \citet{Scheetz-et-al}. The format is a list containing the design matrix which represents the data of $n=120$ rats with $p=200$ gene probes and the response vector with 120 dimensional which represents the expression level of TRIM32 gene.

\begin{table}[!htbp]
\begin{adjustbox}{width=1\textwidth}
\centering
\begin{tabular}{cccccccc}
  \hline
 &$\rm MSE_y$ & $\rm MSE_{\beta}$ & \rm PE &$\rm RMSE_y$ & $\rm RMSE_{\beta}$ & \rm RPE \\ 
  \hline
  GRR & 0.00225(5e-05) & 0.00889(5e-05) & 0.00851(0.00011) & 1.000 & 1.000 & 1.000 \\ 
  RGRR & 0.00078(2e-05) & 0.00238(4e-05) & 0.00592(7e-05) & 2.888 & 3.735 & 1.437 \\ 
  LS & 0.00078(2e-05) & 0.00236(3e-05) & 0.00597(7e-05) & 2.903 & 3.775 & 1.424 \\ 
  PT & 0.00173(5e-05) & 0.00615(0.00015) & 0.00762(0.00011) & 1.298 & 1.447 & 1.117 \\ 
  SPT & 0.00173(5e-05) & 0.00614(0.00015) & 0.00763(0.00011) & 1.302 & 1.449 & 1.115 \\ 
  PS & 0.00078(2e-05) & 0.00237(3e-05) & 0.00599(7e-05) & 2.887 & 3.757 & 1.419 \\ 
  IPT & 0.00078(2e-05) & 0.00237(3e-05) & 0.00597(7e-05) & 2.872 & 3.748 & 1.425 \\ 
  \cmidrule(lr){2-7} 
  Ridge & 0.0054(8e-05) & 0.0246(1e-04) & 0.01134(0.00015) & 0.417 & 0.361 & 0.750 \\ 
  LASSO & 0.00304(8e-05) & 0.01427(1e-04) & 0.00884(0.00014) & 0.741 & 0.623 & 0.963 \\ 
  ALASSO & 0.00304(8e-05) & 0.01289(0.00016) & 0.00919(0.00014) & 0.742 & 0.690 & 0.926 \\ 
  SCAD & 0.00469(0.00016) & 0.01478(0.00036) & 0.01155(0.00023) & 0.480 & 0.601 & 0.737 \\ 
  MCP & 0.00617(0.00022) & 0.02041(0.00039) & 0.01329(3e-04) & 0.365 & 0.436 & 0.640 \\ 
  ENET & 0.00304(8e-05) & 0.01426(1e-04) & 0.00884(0.00014) & 0.741 & 0.623 & 0.963 \\ 
  MNET & 0.00589(2e-04) & 0.02031(0.00041) & 0.01286(0.00028) & 0.382 & 0.438 & 0.661 \\ 
   \hline
\end{tabular}
\end{adjustbox}
\caption{For Eyedata set, the $\rm MSE_y$, $\rm MSE_{\beta}$ and $\rm PE$ and their relative performances based on bootstrap simulation with respect to the GRR.}
\label{tab:eye:res}
\end{table}

We re-sampled bootstrap samples $50$ times. The candidate sub-model is determined by LASSO method, also tuning parameter is selected via ``one standard error" rule. According to this method, the candidate sub-model has 21 important covariates. Similar to first example, we split data in two sets. We calculate $\rm MSE_y$ and $\rm MSE_{\beta}$ and $\rm PE$ based on 5-fold CV for each bootstrap replicate. Again the predictors were first standardized to have zero mean and unit standard deviation, and the response variable is centered before fitting the model. The results are shown in Table \ref{tab:eye:res}. Our results confirm that the suggested estimations outperform the existing penalty methods.

\section{Conclusions}
\label{sec:con}
In this paper, we developed the so-called shrinkage estimators for the regression coefficients using the generalized ridge regression estimator. The general setup included both low-dimensional ($p<n$) and high-dimensional ($p>n$) cases. For our purpose, we first defined a general shrinkage estimator using the Borel measurable function of test statistic for testing $H_0:\bH\bbeta=\boldsymbol{0}$. Then, some specific practical choices considered and their properties obtained. For comparison sake, we also considered some penalty estimators in our study. An extensive Monte Carlo simulation study conducted to compare the performance of the proposed shrinkage ridge estimators with each others and penalty estimators. It can be understood, from the simulation results, that the linear shrinkage estimator surprisingly performs the best in both prediction and error of estimation senses, i.e., it has smallest MSE values, compared to all other estimators. With focus on high-dimensional case, although the penalty estimators do variable selection, but the proposed shrinkage estimators have smaller prediction error. IN conclusion to this, if the purpose of estimation is having smaller MSE rather than variable selection, we suggest to use shrinkage strategies instead of penalty estimators. This result may come to mind a little doubtful. For this reason, we also analyzed the performance of the proposed shrinkage estimators in two real examples to include both low and high-dimensional settings. Nearly the same superior results obtained as discussed in the simulation. According to the results of Tables 4 \& 5, in the performance comparison between shrinkage and penalty estimators, it is seen that the proposed shrinkage estimators are superior when we combine the information of full-model and sub-model. 

\section{Appendix}
\begin{proof}[Proof of Theorem~\ref{bias}]

Since all of the pronounced estimators is a special case of $\hbbeta_n^{\rm Shrinkage}$, we give the bias of this estimator here. Then, the proof follows by applying relevant $g(\cdot)$ function in each estimator. Hence, we have
\begin{eqnarray}\label{eq31}
{\rm\bB}(\hbbeta_n^{\rm Shrinkage})&=&{\rm\bB}(\hbbeta_n^{\rm GRR})-\mathbb{E}\left[(\hbbeta_n^{\rm GRR}-(\hbbeta_n^{\rm RGRR})')g(\mathscr{W}_{n})\right]\cr
&=&{\rm\bB}(\hbbeta_n^{\rm GRR})-(I_p-\bM_{\bK})\mathbb{E}\left[\hbbeta_n^{\rm GRR}g(\mathscr{W}_{n})\right]
\end{eqnarray}
Since $\hbbeta^{\rm GRR}=\bS_{\bK}^{-1}\bS\bS^{-1}\bX'\bY=\bS_{\bK}^{-1}\bS\hbbeta^{\rm LSE}$, the bias is obtained directly using $\mathbb{E}(\hbbeta^{\rm LSE})=\bbeta$ as ${\rm\bB}(\hbbeta_n^{\rm GRR})=(\bS_{\bK}^{-1}\bS-\bI_p)\bbeta$. Using Theorem 1 in Appendix B of \citet{JB1978}, 
\begin{eqnarray}\label{eq32}
\mathbb{E}\left[\hbbeta_n^{\rm GRR}g(\mathscr{W}_{n})\right]&=&\bS_{\bK}^{-1}\bS
\mathbb{E}\left[\hbbeta_n^{\rm LSE}g(\mathscr{W}_{n})\right]\cr
&=&\bS_{\bK}^{-1}\bS\bbeta\mathbb{E}\left[g(\mathscr{F}_{q+2,m}(\Delta^2))\right]
\end{eqnarray}
Substituting \eqref{eq32} in \eqref{eq31} gives
\begin{eqnarray*}
{\rm\bB}(\hbbeta_n^{\rm Shrinkage})&=&(\bS_{\bK}^{-1}\bS-\bI_p)\bbeta-(\bI_p-\bM_{\bK})\bS_{\bK}^{-1}\bS\bbeta\mathbb{E}\left[g(\mathscr{F}_{q+2,m}(\Delta^2))\right]
\end{eqnarray*}

\end{proof}

\begin{proof}[Proof of Theorem~\ref{risk}]

Similar to the proof of Theorem \ref{bias}, we provide the quadratic risk of the shrinkage estimator $\hbbeta_n^{\rm Shrinkage}$ here. Then, the proof follows by applying relevant $g(\cdot)$ function in each estimator. Hence, we have
\begin{eqnarray}\label{eq34}
{\rm R}(\hbbeta_n^{\rm Shrinkage})&=&
{\rm R}(\hbbeta_n^{\rm GRR})-2\mathbb{E}\left[(\hbbeta_n^{\rm GRR}-\bbeta)'\bW(\hbbeta_n^{\rm GRR}-\hbbeta_n^{\rm RGRR})g(\mathscr{W}_n)\right]\cr
&&+\mathbb{E}\left[(\hbbeta_n^{\rm GRR}-\hbbeta_n^{\rm RGRR})'\bW(\hbbeta_n^{\rm GRR}-\hbbeta_n^{\rm RGRR})g^2(\mathscr{W}_n)\right].
\end{eqnarray}
where ${\rm R}(\hbbeta_n^{\rm GRR})=\mathbb{E}\left[(\hbbeta_n^{\rm GRR}-\bbeta)'\bW(\hbbeta_n^{\rm GRR}-\bbeta)\right]$. 
Using $\hbbeta^{\rm GRR}=\bS_{\bK}^{-1}\bS\hbbeta^{\rm LSE}$, we have
\begin{eqnarray}\label{eq35}
{\rm R}(\hbbeta_n^{\rm GRR})&=&\mathbb{E}\left[(\hbbeta_n^{\rm LSE}-\bbeta)'\bS\bS_{\bK}^{-1}\bW\bS_{\bK}^{-1}\bS(\hbbeta_n^{\rm LSE}-\bbeta)\right]\cr
&&+2\mathbb{E}\left[(\hbbeta_n^{\rm LSE}-\bbeta)'\bS\bS_{\bK}^{-1}\bW(\bS_{\bK}^{-1}\bS-\bI_p)\bbeta\right]\cr
&&+\bbeta'(\bS_{\bK}^{-1}\bS-\bI_p)\bW(\bS\bS_{\bK}^{-1}-\bI_p)\bbeta\cr
&=&\sigma^2{\rm tr}\left(\bS_{\bK}^{-1}\bS\bS_{\bK}^{-1}\bW\right)+\bbeta'(\bS_{\bK}^{-1}\bS-\bI_p)\bW(\bS\bS_{\bK}^{-1}-\bI_p)\bbeta
\end{eqnarray}
On the other hand, using Theorems 1 \& 3 in Appendix B of \citet{JB1978}, yields
\begin{eqnarray}\label{eq36}
\mathbb{E}\left[(\hbbeta_n^{\rm GRR}-\bbeta)'\bW(\hbbeta_n^{\rm GRR}-(\hbbeta_n^{\rm RGRR})')g(\mathscr{W}_n)\right]&=&\mathbb{E}\left[(\hbbeta_n^{\rm LSE})'\bB\hbbeta_n^{\rm LSE}g(\mathscr{W}_n)\right]\cr
&&-\bbeta'\bW(\bI_p-\bM_{\bK})\bS_{\bK}^{-1}\bS\mathbb{E}\left[\hbbeta_n^{\rm LSE}g(\mathscr{W}_n)\right]\cr
&=&{\rm tr}(\bB)\mathbb{E}[g(F_{q+2,m}(\Delta^2))]\cr
&&+\bbeta'\bB\bbeta\mathbb{E}[g(F_{q+4,m}(\Delta^2))]\cr
&&-\bbeta'\bW(\bI_p-\bM_{\bK})\bS_{\bK}^{-1}\bS\bbeta\mathbb{E}\left[g(F_{q+2,m}(\Delta^2))\right]
\end{eqnarray}
where $\bB=\bS\bS_{\bK}^{-1}\bW(\bI_p-\bM_{\bK})\bS_{\bK}^{-1}\bS$. Finally
\begin{eqnarray}\label{eq37}
\mathbb{E}\left[(\hbbeta_n^{\rm GRR}-(\hbbeta_n^{\rm RGRR})')'\bW(\hbbeta_n^{\rm GRR}-(\hbbeta_n^{\rm RGRR})')g^2(\mathscr{W}_n)\right]&=&\mathbb{E}\left[(\hbbeta_n^{\rm LSE})'\bC\hbbeta_n^{\rm LSE}g^2(\mathscr{W}_n)\right]\cr
&=&{\rm tr}(\bC)\mathbb{E}[g^2(F_{q+2,m}(\Delta^2))]\cr
&&+\bbeta'\bC\bbeta\mathbb{E}[g^2(F_{q+4,m}(\Delta^2))]
\end{eqnarray}
where $\bC=\bS\bS_{\bK}^{-1}(\bI_p-\bM_{\bK})\bW(\bI_p-\bM_{\bK})\bS_{\bK}^{-1}\bS$.

Combining \eqref{eq35}-\eqref{eq37}, gives
\begin{eqnarray*}
{\rm R}(\hbbeta_n^{\rm Shrinkage})&=&\sigma^2{\rm tr}\left(\bS_{\bK}^{-1}\bS\bS_{\bK}^{-1}\bW\right)+\bbeta'(\bS_{\bK}^{-1}\bS-\bI_p)\bW(\bS\bS_{\bK}^{-1}-\bI_p)\bbeta\cr
&&-2{\rm tr}(\bB)\mathbb{E}[g(F_{q+2,m}(\Delta^2))]-2\bbeta'\bB\bbeta\mathbb{E}[g(F_{q+4,m}(\Delta^2))]\cr
&&+2\bbeta'\bW(\bI_p-\bM_{\bK})\bS_{\bK}^{-1}\bS\bbeta\mathbb{E}\left[g(F_{q+2,m}(\Delta^2))\right]\cr
&&+{\rm tr}(\bC)\mathbb{E}[g^2(F_{q+2,m}(\Delta^2))]+\bbeta'\bC\bbeta\mathbb{E}[g^2(F_{q+4,m}(\Delta^2))]
\end{eqnarray*}
\end{proof}

\section*{Acknowledgments}
Second author Mohammad Arashi's work is based on the research supported in part by the National Research Foundation of South Africa (Grant NO. 109214). Third author S. Ejaz Ahmed is supported by the Natural Sciences and the Engineering Research Council of Canada (NSERC).



\begin{thebibliography}{999}

\bibitem[Ahmed, 1992]{Ahmed1992}
Ahmed, S. E. (1992). Shrinkage preliminary test estimation in multivariate normal distributions. Journal of Statistical Computation and Simulation, 43:177--195.

\bibitem[Ahmed, 2014]{Ahmed2014}
Ahmed, S. E. (2014). Penalty, shrinkage and pretest strategies: Variable selection and estimation. Springer International Publishing.

\bibitem[Ahmed \& Y\"{u}zba\c{s}{\i}, 2017]{AhmedYzb2017}
Ahmed, S. E., \& Y\"{u}zba\c{s}{\i}, B. (2017). High Dimensional Data Analysis: Integrating Submodels. In Big and Complex Data Analysis (pp. 285-304). Springer International Publishing.

\bibitem[Ahmed \& Y\"{u}zba\c{s}{\i}, 2016]{AhmedYzb2016}
Ahmed, S. E., \& Y\"{u}zba\c{s}{\i}, B. (2016). Big data analytics: integrating penalty strategies. International Journal of Management Science and Engineering Management, 11(2), 105-115.










\bibitem[Efron et al., 2004]{Efron2004}
Efron, B., Hastie, T., Johnstone, I., \& Tibshirani, R. (2004). Least angle regression. The Annals of statistics, 32(2), 407-499.



\bibitem[Frank \& Friedman, 1993]{Frank1993}
Frank, I. and Friedman, J. (1993). A statistical view of some chemometrics regression tools (with discussion). Technometrics, 35, 109 -- 148.



\bibitem[Fan \& Li, 2001]{FanLi2001}
Fan, J., \& Li, R. (2001). Variable selection via nonconcave penalized likelihood and its oracle properties. Journal of the American statistical Association, 96(456), 1348-1360.


\bibitem[Gao et al, 2017]{gao-et-al2017}
Gao, X., Ahmed, S. E., \& Feng, Y. (2017). Post selection shrinkage estimation for high--dimensional data analysis. Applied Stochastic Models in Business and Industry, 33(2), 97-120.




\bibitem[Hoerl \& Kennard, 1970]{Ho-Ke1970}
Hoerl, A.E. and Kennard,~R.W. (1970). Ridge Regression: Biased estimation for non-orthogonal problems., Technometrics, 12, 69 -- 82.



\bibitem[Huang et al, 2010]{Huangetal2010}
Huang, J., Breheny, P., Ma, S., and Zhang, C. H. (2010). The MNET method for variable selection. Unpublished) Technical Report, (402).
ISO 690	






\bibitem[Ishwaran \& Rao, 2014]{Ishwaran_Rao2014} 
Ishwaran, H., \& Rao, J. S. (2014). Geometry and properties of generalized ridge regression in high dimensions. Contemp. Math, 622, 81-93.

\bibitem[Judge \& Bock, 1978]{JB1978} 
Judge, G. G., \& Bock, M. E. (1978). The statistical implicatinos of pre-test and stein-rule estimators in econometrics.









\bibitem[Lan et al., 2014]{Lan-et-al2014}
Lan, W., Wang, H., \& Tsai, C. L. (2014). Testing covariates in high-dimensional regression. Annals of the Institute of Statistical Mathematics, 66(2), 279-301.








\bibitem[McDonald \& Schwing, 1973]{McSc1973}
McDonald, G. C., \& Schwing, R. C. (1973). Instabilities of regression estimates relating air pollution to mortality. Technometrics, 15(3), 463-481.




\bibitem[Roozbeh \& Arashi, 2016]{rozAras2016}
 Roozbeh, M. and Arashi, M. (2016). Shrinkage ridge regression in partial linear models, Comm. Statist. Sim. Comp., 45(20), 6022-6044.





\bibitem[Saleh, 2006]{saleh2006}
Saleh. A. K. Md. Ehsanes. (2006). Theory of Preliminary Test and Stein-Type Estimation with Applications, Wiley; United Stated of America.


\bibitem[Scheetz et al., 2006]{Scheetz-et-al}
Scheetz, T. E., Kim, K. Y. A., Swiderski, R. E., Philp, A. R., Braun, T. A., Knudtson, K. L., ... \& Sheffield, V. C. (2006). Regulation of gene expression in the mammalian eye and its relevance to eye disease. Proceedings of the National Academy of Sciences, 103(39), 14429-14434.









\bibitem[Tibshirani, 1996]{Tibshirani1996}
Tibshirani, R. (1996). Regression shrinkage and selection via the Lasso. Journal of the Royal Statistical Society. Series B (Methodological), 267-288.








\bibitem[Tutz \& Ulbricht, 2009]{tutz:ulbricht2009}
Tutz, G., \& Ulbricht, J. (2009). Penalized regression with correlation-based penalty. Statistics and Computing, 19(3), 239-253.





\bibitem[Wu \& Asar, 2016]{wuasar2016}
Wu, J., \& Asar, Y. (2016). A weighted stochastic restricted ridge estimator in partially linear model. Comm. Statist. Theory Methods, 
DOI: 10.1080/03610926.2016.1206936

\bibitem[Y\"{u}zba\c{s}{\i} et al., 2017b]{yuzbasi-et-al2017b}
Y\"{u}zba\c{s}{\i}, B., Arashi, M., \& Ahmed, S. E. (2017). Big Data Analysis Using Shrinkage Strategies. arXiv preprint arXiv:1704.05074.


\bibitem[Y\"{u}zba\c{s}{\i} et al., 2017a]{yuzbasi-et-al2017}
Y\"{u}zba\c{s}{\i}, B., Ahmed, S.E. and Gungor, M., Improved Penalty Strategies in Linear Regression Models, \textit{REVSTAT--Statistical Journal}, 15(2)(2017), 251--276.


\bibitem[Y\"{u}zba\c{s}{\i} \& Ahmed, 2016]{bahadir-ahmed2016}
Y\"{u}zba\c{s}{\i}, B., \& Ejaz Ahmed, S. (2016). Shrinkage and penalized estimation in semi-parametric models with multicollinear data. Journal of Statistical Computation and Simulation, 1-19.

\bibitem[Y\"{u}zba\c{s}{\i} \& Ahmed, 2015]{bahadir-ahmed2015}
Y\"{u}zba\c{s}{\i}, B., \& Ahmed, S. E. (2015). Shrinkage Ridge Regression Estimators in High-Dimensional Linear Models. In Proceedings of the Ninth International Conference on Management Science and Engineering Management (pp. 793 -- 807). Springer Berlin Heidelberg.


\bibitem[Zhang, 2007]{Zhang2007}
Zhang, C.H. (2007). Penalized Linear Unbiased Selection. Rutgers University, Department of Statistics and Biostatistics Technical Report, 2007-003.


\bibitem[Zhang, 2010]{Zhang2010}
Zhang, C.H. (2010). Nearly unbiased variable selection under minimax concave penalty. Annals os Statistics, 38, 894 -- 942.




\bibitem[Zou, 2006]{Zou2006}
Zou, H. (2006). The adaptive LASSO and its oracle properties. Journal of the American statistical association, 101(476), 1418-1429.


\bibitem[Zou and Hastie, 2005]{ZouHastie2005}
Zou, H. and Hastie, T. (2005). Regularization and Variable Selection via the Elastic Net. Journal of the Royal Statistical Society B, 67: 301-320.








\end{thebibliography}
\end{document}